\documentclass[a4paper]{article}
\usepackage[utf8]{inputenc}
\usepackage[T1]{fontenc}

\usepackage{amsmath,amssymb,amsthm}
\usepackage{mathtools,mathdots}
\mathtoolsset{centercolon}
\usepackage{dsfont}

\usepackage[pagebackref, colorlinks=true, allcolors=blue, pdfborder={0 0 0}]{hyperref}
\usepackage[a4paper]{geometry}

\usepackage[textsize=footnotesize,textwidth=17ex]{todonotes}

\usepackage{multicol}
\usepackage{enumitem}
\setlist{noitemsep}
\setlist[1]{labelindent=\parindent, topsep=0.5\topsep}
\setlist[enumerate,1]{label=\textnormal{(\hspace{-0.2ex}\textit{\roman*}\hspace{-0.1ex})}}
\makeatletter
\newcommand{\mylabel}[2]{#2\def\@currentlabel{#2}\label{#1}}
\makeatother

\usepackage{fancyhdr}
{%
	\fancyhf{}
	\fancyhead[L]{\textsc{\nouppercase{\leftmark}}}
	\fancyhead[R]{\textsc{\textbf{\thepage}}}

}

\newtheorem{theorem}{Theorem}[section]
\newtheorem{lemma}[theorem]{Lemma}
\newtheorem{proposition}[theorem]{Proposition}
\newtheorem{corollary}[theorem]{Corollary}
\theoremstyle{definition}
\newtheorem{definition}[theorem]{Definition}

\newtheorem{examples}[theorem]{Examples}
\newtheorem{notation}[theorem]{Notation}
\newtheorem{remark}[theorem]{Remark}
\newtheorem{construction}{Construction}

\newcommand{\M}{\mathbb{M}}
\newcommand{\Z}{\mathbb{Z}}
\newcommand{\N}{\mathbb{N}}
\renewcommand{\P}{\mathbb{P}}
\newcommand{\id}{\mathds{1}}
\newcommand{\PSL}{\mathsf{PSL}}
\newcommand{\End}{\mathsf{End}}

\renewcommand{\leq}{\leqslant}

\renewcommand{\phi}{\varphi}
\renewcommand{\nsim}{\not\sim}
\newcommand*{\lsim}{\mathord{\sim}}
\newcommand{\til}{\mathord{\sim}}

\newcommand\cdottil{\mathchoice
	{\mathrel{\rlap{\raisebox{-0.5ex}{$\displaystyle\tilde{\phantom{\cdot}}$}}\displaystyle\cdot}}
	{\mathrel{\rlap{\raisebox{-0.5ex}{$\textstyle\tilde{\phantom{\cdot}}$}}\textstyle\cdot}}
	{\mathrel{\rlap{\raisebox{-0.35ex}{$\scriptstyle\tilde{\phantom{\cdot}}$}}\scriptstyle\cdot}}
	{\mathrel{\rlap{\raisebox{-0.25ex}{$\scriptscriptstyle\tilde{\phantom{\cdot}}$}}\scriptscriptstyle\cdot}}}
\newcommand\plustil{\mathrel{\mathchoice
	{\rlap{\raisebox{-0.3ex}{$\displaystyle\tilde{\phantom{+}}$}}\displaystyle +}
	{\rlap{\raisebox{-0.3ex}{$\textstyle\tilde{\phantom{+}}$}}\textstyle +}
	{\rlap{\raisebox{-0.2ex}{$\scriptstyle\tilde{\phantom{+}}$}}\scriptstyle +}
	{\rlap{\raisebox{-0.1ex}{$\scriptscriptstyle\tilde{\phantom{+}}$}}\scriptscriptstyle +}}}
\newcommand\mintil{\mathrel{\mathchoice
	{\rlap{\raisebox{-0.8ex}{$\displaystyle\tilde{\phantom{-}}$}}\displaystyle -}
	{\rlap{\raisebox{-0.8ex}{$\textstyle\tilde{\phantom{-}}$}}\textstyle -}
	{\rlap{\raisebox{-0.6ex}{$\scriptstyle\tilde{\phantom{-}}$}}\scriptstyle -}
	{\rlap{\raisebox{-0.5ex}{$\scriptscriptstyle\tilde{\phantom{-}}$}}\scriptscriptstyle -}}}
\newcommand\mutil{\tilde{\mu}}

\DeclareMathOperator{\im}{Im}
\DeclareMathOperator{\Sym}{Sym}
\DeclareMathOperator{\Hom}{\mathsf{Hom}}
\DeclarePairedDelimiter{\abs}{\lvert}{\rvert}
\DeclareMathOperator{\Rad}{Rad}

\newcommand{\dash}{\nobreakdash-\hspace{0pt}}

\numberwithin{equation}{section}

\title{Local Moufang sets and local Jordan pairs}
\author{Tom De Medts (\href{mailto:Tom.DeMedts@UGent.be}{Tom.DeMedts@UGent.be}) \\
Department of Mathematics, Ghent University\\
Krijgslaan 281 -- S22, B-9000 Gent, Belgium \\[1em]
Erik Rijcken\footnote{PhD Fellow of the Research Foundation - Flanders (Belgium) (F.W.O.-Vlaanderen)}\ \  (\href{mailto:erijcken@cage.ugent.be}{erijcken@cage.ugent.be})\\
Department of Mathematics, Ghent University\\
Krijgslaan 281 -- S22, B-9000 Gent, Belgium}

\begin{document}

\maketitle

\begin{abstract}
	In this paper, we extend the theory of special local Moufang sets. We construct a local Moufang set from every local Jordan pair, and we show that every local Moufang set satisfying certain (natural) conditions gives rise to a local Jordan pair. We also explore the connections between these two constructions.
\end{abstract}

\section{Introduction}

In a previous paper, we have introduced \emph{local Moufang sets} as generalizations of Moufang sets. Moufang sets were introduced by Jacques Tits in \cite{MR1200265} and provide a method to describe many groups of algebraic origin of rank one,
including all linear algebraic groups defined over an arbitrary field $k$ and of $k$-rank equal to $1$.
The anisotropic kernel of such a group and the Moufang set corresponding to the group can typically be described in terms of an anistropic algebraic structure:
a field, a Jordan division algebra \cite{MR2221120}, or, more generally, a structurable division algebra \cite{BDS}.

In the same spirit, local Moufang sets give a framework to describe many of these groups over local rings instead of fields.
In our previous paper \cite{DMRijckenPSL}, we explored some basic theory of local Moufang sets and we investigated the structure of $\PSL_2(R)$ for local rings $R$.

It has been known for some time that every Jordan division algebra gives rise to a Moufang set \cite{MR2221120}.
These Moufang sets have been characterized as the special Moufang sets with abelian root groups satisfying some mild linearity condition \cite[Theorem~1.2]{MR2425693}.

The current paper expands the theory of special local Moufang sets.
Rather than working with local Jordan algebras, we use (local) Jordan pairs, first introduced by K. Meyberg in \cite{MR0263883} and studied extensively by O. Loos in \cite{MR0444721}. This allows us to use the two-sided structure of the local Moufang sets, and avoids the explicit choice of a unit element. As we assume the Jordan pairs to have at least one invertible element, the Jordan pairs correspond one-to-one to Jordan algebras.
We show that every local Jordan pair gives rise to a local Moufang set, and we characterize these local Moufang sets, very much in the style of \cite[Theorem~1.2]{MR2425693} mentioned above.

\paragraph*{Outline of the paper.}

In section~\ref{sec:LMS}, we summarize the relevant definitions, properties and theorems from the general theory of local Moufang sets as developed in \cite{DMRijckenPSL}. Section~\ref{sec:SpecialLMS} expands on the theory of special local Moufang sets, in particular those with abelian root groups. One important proposition here is a sufficient condition for the root groups to be uniquely $k$\dash divisible (Proposition \ref{prop:ndivglobal}). We also prove a few simple-looking identities in Proposition~\ref{prop:mu-involution}, which will be surprisingly crucial in proving the main result of section~\ref{sec:LMStoJP}.

In section~\ref{sec:JP}, we first recall the definition of a local Jordan pair, along with some relevant properties.
We then proceed to describe how to construct a local Moufang set from a local Jordan pair, using a construction from \cite{DMRijckenPSL}.

In section~\ref{sec:LMStoJP}, our aim is to reverse this construction:
given a local Moufang set satisfying properties \hyperref[itm:J1]{\normalfont{(J1-4)}}, Construction~\ref{constr:Jpair} gives an algebraic structure that will turn out to be a local Jordan pair. We prove many intermediate identities before we can finally conclude that we indeed get a local Jordan pair (Theorem~\ref{thm:localJP}).

In the final section, we connect sections~\ref{sec:JP} and \ref{sec:LMStoJP}. In particular, we show that if we start with a local Jordan pair, construct the local Moufang set, and use this local Moufang set to construct a new local Jordan pair, then this Jordan pair is isomorphic to the Jordan pair we started with.
Conversely, if we have a local Moufang set satisfying the assumptions required to construct a local Jordan pair, we would like this local Moufang set to be isomorphic to the one we construct from the Jordan pair. In order to prove this, we need one additional assumption on the local Moufang set (Theorem~\ref{th:extra}).
This last result provides a characterization of the local Moufang sets arising from local Jordan pairs.

\paragraph*{Acknowledgment.}

We are grateful to Ottmar Loos and Holger Petersson for sharing their insight in the examples of local Jordan algebras
(see Examples~\ref{ex:local} below).

\section{Local Moufang sets}\label{sec:LMS}

\subsection{Definition and conventions}

In \cite{DMRijckenPSL}, we introduced local Moufang sets as a generalization of Moufang sets, in order to study groups of rank one over a local structure. For the reader's convenience, we recall the necessary definitions and notations we will use throughout this article.
We will, of course, often refer to {\em loc.\@~cit.} for the proofs of the facts we will be using.

\begin{notation}\leavevmode
	\begin{itemize}
		\item If $(X,\lsim)$ is a set with an equivalence relation, we denote the equivalence class of $x\in X$ by~$\overline{x}$, and the set of equivalence classes by $\overline{X}$.
		\item We denote the group of equivalence-preserving permutations of $X$ by $\Sym(X,\lsim)$.
		\item If $g \in \Sym(X,\lsim)$, we will denote the corresponding element of $\Sym(\overline{X})$ by $\overline{g}$.
		\item Our actions will always be on the right. The action of an element $g$ on an element $x$ will be denoted by $x\cdot g$ or $xg$. Conjugation will correspondingly be denoted by $g^h = h^{-1}gh$.
	\end{itemize}
\end{notation}

\begin{definition}
	A \emph{local Moufang set} $\M$ consists of a set with an equivalence relation $(X,\lsim)$ such that $\abs{\overline{X}}>2$, and a family of subgroups $U_x\leq\Sym(X,\lsim)$ for all $x\in X$, called the \emph{root groups}. We denote $U_{\overline{x}}:=\overline{U_x}=\im(U_x\to \Sym(\overline{X}))$ for the permutation group induced by the action of $U_x$ on the set of equivalence classes.
    (This notation is justified by \ref{itm:LM1} below.)
    The group generated by the root groups is called the \emph{little projective group}, and will usually be denoted by $G:=\langle U_x\mid x\in X\rangle$. Furthermore, we demand the following:
	\begin{enumerate}[label=\textnormal{(LM\arabic*)},leftmargin=10ex]
		\item If $x\sim y$ for $x,y\in X$, then $U_{\overline{x}}=U_{\overline{y}}$\label{itm:LM1}.
		\item For $x\in X$, $U_x$ fixes $x$ and acts sharply transitively on $X\setminus\overline{x}$\label{itm:LM2}.
		\item[\mylabel{itm:LM2'}{\textnormal{(LM2')}}] For $\overline{x}\in \overline{X}$, $U_{\overline{x}}$ fixes $\overline{x}$ and acts sharply transitively on $\overline{X}\setminus\{\overline{x}\}$.
		\item For $x\in X$ and $g\in G$, we have $U_x^g = U_{xg}$\label{itm:LM3}.
	\end{enumerate}
\end{definition}

These axioms imply in particular that $(\overline{X},\{U_{\overline{x}}\}_{\overline{x}\in\overline{X}})$ is a Moufang set.

\begin{definition}
	Two local Moufang sets $\M$ and $\M'$ are \emph{isomorphic}, denoted $\M\cong\M'$, if there is a bijection $\phi \colon X\to X'$ and group isomorphisms $\theta_x  \colon  U_x\to U'_{\phi(x)}$ such that
	\begin{itemize}
		\item for all $x,y\in X$, we have $x\sim y \iff \phi(x)\sim'\phi(y)$;
		\item for all $x,y\in X$ and $u\in U_y$, we have $\phi(x\cdot u) = \phi(x)\cdot \theta_y(u)$.
	\end{itemize}
\end{definition}

The next proposition roughly states that root groups of two non-equivalent points already contain all the information of a local Moufang set, and that any such two non-equivalent points play the same role.

\begin{proposition}\label{pr:UxUy}
	Let $\M$ be a local Moufang set, and $x,y\in X$ with $x\nsim y$. Then $\langle U_x,U_y \rangle = G$. The little projective group $G$ acts transitively on $\{(x,y)\in X^2\mid x\nsim y\}$.
\end{proposition}
\begin{proof}
	This is \cite[Proposition 2.4 and 2.6]{DMRijckenPSL}.
\end{proof}

\begin{notation}\leavevmode
	\begin{itemize}
		\item In a local Moufang set, we fix two points of $X$ that are not equivalent, and we call them $0$ and $\infty$.
		\item For any $x\nsim\infty$, by \ref{itm:LM2}, there is a unique element of $U_\infty$ mapping $0$ to $x$. We denote this element by $\alpha_x$. In particular, $\alpha_0 = \id$.
		\item For $x\nsim\infty$, we set $-x:=0\cdot\alpha_x^{-1}$, so $\alpha_x^{-1} = \alpha_{-x}$.
	\end{itemize}
\end{notation}

\subsection{Units and \texorpdfstring{$\mu$}{mu}-maps}

When we have fixed $0$ and $\infty$, there will be many elements of $X$ that behave more nicely than others. These are precisely the elements of $X$ that do not project to $\overline{0}$ or $\overline{\infty}$ in $\overline{X}$.

\begin{definition}
	In a local Moufang set, an element $x\in X$ is a \emph{unit} if $x\nsim0$ and $x\nsim\infty$.
\end{definition}

\begin{proposition}\label{pr:unit}
	Let $\M$ be a local Moufang set, and $x\in X$ with $x\nsim\infty$. Then the following are equivalent:
	\begin{enumerate}
		\item $x$ is a unit\label{itm:unit1};
		\item $\overline{\alpha_x}$ does not fix $\overline{0}$\label{itm:unit2};
		\item $\overline{\alpha_x}$ does not fix any element of $\overline{X} \setminus \overline{\infty}$\label{itm:unit3}.
	\end{enumerate}
	In particular, $x$ is a unit if and only if $-x$ is a unit.
\end{proposition}
\begin{proof}
	This is \cite[Proposition 2.9]{DMRijckenPSL}.
\end{proof}

Each of these units admits an element $\mu_x$ interchanging $0$ and $\infty$:

\begin{proposition}\label{pr:mu}
	For each unit $x\in X$, there is a unique element  $\mu_x\in U_0\alpha_x U_0$ such that $0\mu_x=\infty$ and $\infty\mu_x=0$;
    it is called the \emph{$\mu$-map} corresponding to $x$.
    Moreover, $\mu_x = g\alpha_x h$ with $g$ the unique element of $U_0$ mapping $\infty$ to $-x$ and $h$ the unique element of $U_0$ mapping $x$ to $\infty$.
\end{proposition}
\begin{proof}
	This is \cite[Proposition 2.12]{DMRijckenPSL}.
\end{proof}

\begin{notation}\leavevmode
	\begin{itemize}
		\item We fix one $\mu$-map and call it $\tau$. Recall that $\abs{\overline{X}}>2$, so there is at least one unit.
		\item For each $x\nsim\infty$, we set $\gamma_x := \alpha_x^\tau\in U_0$, which is the unique element of $U_0$ mapping $\infty$ to~$x\tau$.
        \item For each unit $x$, we set $\til x := (-(x\tau^{-1}))\tau$.
	\end{itemize}
\end{notation}

\begin{lemma} \label{prop:mu}Let $x$ be a unit. %
	\begin{enumerate}[topsep=0pt]
		\item Let $y\in X$. Then $y$ is a unit if and only if $y\mu_x$ is a unit;
		\item $\mu_x$ does not depend on the choice of $\tau$; \label{itm:muindep}
		\item $\mu_{-x} = \mu_x^{-1}$;
		\item $\mu_{x\tau} = \mu_{-x}^\tau$; \label{itm:mutau}
		\item $\mu_{x} = \alpha_x\alpha_{-(x\tau^{-1})}^\tau\,\alpha_{-\til x}$; \label{itm:muform2}
		\item $\til x = -((-x)\mu_x)$;\label{itm:tilmu}
		\item $\til x$ does not depend on the choice of $\tau$; \label{itm:tiltau}
		\item $\mu_{-x} = \alpha_{-\til x}\mu_{-x}\alpha_x\mu_{-x}\alpha_{\til -x}$. \label{itm:muform3}
	\end{enumerate}
\end{lemma}
\begin{proof}
	This is \cite[Lemma 2.14(3) and 2.17]{DMRijckenPSL}.
\end{proof}

In the following proposition, we added an extra statement compared to Proposition~2.18 of \cite{DMRijckenPSL}. This new identity will be simplified in the case of special local Moufang sets with abelian root groups, and will then be a starting point for proving one of the axioms of Jordan pairs in Proposition~\ref{prop:JMS_JP1_units}.

\begin{proposition}\label{prop:sumform}
	Let $x,y\in X$ be units such that $x\nsim y$ and set $z := x\tau^{-1}\alpha_{-(y\tau^{-1})}\tau$.
	\begin{enumerate}[topsep=0pt]
		\item $z$ is independent of the choice of $\tau$;
		\item $z = x\alpha_{-y}\mu_y\alpha_{\til y}$ and $\til z = y\alpha_{-x}\mu_x\alpha_{\til x}$;
		\item $\mu_y\mu_z\mu_{-x} = \mu_{y\alpha_{-x}}$.
	\end{enumerate}
\end{proposition}
\begin{proof}\leavevmode
	\begin{enumerate}
		\item[(\textit{i-ii})] This is \cite[Proposition 2.18]{DMRijckenPSL}.
		\item[(\textit{iii})] We repeatedly use Lemma~\ref{prop:mu}\ref{itm:muform2}, where we interchanged $\tau$ and $\tau^{-1}$.
	\begin{align*}
		\mu_z &= \alpha_z\alpha_{-(z\tau)}^{\tau^{-1}}\alpha_{-\til z}
			= \alpha_z\alpha_{y\tau}^{\tau^{-1}}\alpha_{-(x\tau)}^{\tau^{-1}}\alpha_{-\til z} \\
			&= \alpha_z\alpha_{-\til y}\mu_{-y}\alpha_y\alpha_{-x}\mu_x\alpha_{\til x}\alpha_{-\til z} \\
			&= \alpha_{x\alpha_{-y}\mu_y}\mu_{-y}\alpha_y\alpha_{-x}\mu_x\alpha_{-(y\alpha_{-x}\mu_x)}
	\intertext{Hence, again using Lemma~\ref{prop:mu}\ref{itm:muform2} but now with $\mu_x$ and $\mu_y$, we get}
		\mu_y\mu_z\mu_{-x} &= \mu_y\alpha_{x\alpha_{-y}\mu_y}\mu_{-y}\alpha_y\alpha_{-x}\mu_x\alpha_{-(y\alpha_{-x}\mu_x)}\mu_{-x} \\
			&= \alpha_{-\til(x\alpha_{-y})}\mu_{y\alpha_{-x}}\alpha_{x\alpha_{-y}}\alpha_y\alpha_{-x}\alpha_{-(y\alpha_{-x})}\mu_{y\alpha_{-x}}\alpha_{\til(y\alpha_{-x})} \\
			&= \alpha_{-\til(x\alpha_{-y})}\mu_{-(x\alpha_{-y})}\alpha_{x\alpha_{-y})}\mu_{-(x\alpha_{-y})}\alpha_{\til-(x\alpha_{-y})} \\
			&= \mu_{-(x\alpha_{-y})} = \mu_{y\alpha_{-x}}\;,
	\end{align*}
	using Lemma~\ref{prop:mu}\ref{itm:muform3}.\qedhere
	\end{enumerate}
\end{proof}

\begin{lemma}\label{prop:huaAut}
	Let $\tau$ and $\mu$ be any two $\mu$-maps. Then for any $x\in X\setminus\overline{\infty}$, we have $\alpha_x^{\tau\mu} = \alpha_{x\tau\mu}$.
	In particular, $(\alpha_x\alpha_y)^{\tau\mu} = \alpha_{x\tau\mu}\alpha_{y\tau\mu}$ for all $x,y\in X\setminus\overline{\infty}$.
\end{lemma}
\begin{proof}
	This is \cite[Lemma~2.23]{DMRijckenPSL}, using the fact that $\tau\mu$ is a Hua map (\cite[Definition~2.19]{DMRijckenPSL}).
\end{proof}

\subsection{Constructing local Moufang sets}

We already observed that the root groups of two non-equivalent points contain all the data of a local Moufang set. In fact, only one root group and one $\mu$-map are needed. The following construction starts with a group $U$ and an element $\tau$ satisfying some basic properties, and creates all the data for a local Moufang set.

\begin{construction}\label{constr:MUtau}
	The construction requires some data to start with. We need
	\begin{itemize}
		\item a set with an equivalence relation $(X,\lsim)$, such that $\abs{\overline{X}}>2$;
		\item a group $U\leq\Sym(X,\lsim)$;
        \item an element $\tau\in\Sym(X,\lsim)$.
	\end{itemize}
	The action of $U$ and $\tau$ will have to be sufficiently nice in order to do the construction, so we demand that
	\begin{enumerate}[label=\textnormal{(C\arabic*)},leftmargin=1cm]
		\item $U$ has a fixed point which we call $\infty$, and acts sharply transitively on $X\setminus\overline{\infty}$\label{itm:C1};
		\item[\mylabel{itm:C1'}{\textnormal{(C1')}}] the induced action of $U$ on $\overline{X}$ is sharply transitive on $\overline{X}\setminus\{\overline{\infty}\}$;
		\item $\infty\tau\nsim\infty$ and $\infty\tau^2=\infty$; we write $0:=\infty\tau$\label{itm:C2}.
	\end{enumerate}
	In this construction, we now define the following objects:
	\begin{itemize}
		\item For $x\nsim\infty$, we let $\alpha_x$ be the unique element of $U$ mapping $0$ to $x$ (by \ref{itm:C1} and \ref{itm:C2}).
		\item For $x\nsim\infty$, we write $\gamma_x:=\alpha_x^\tau$, which then maps $\infty$ to $x\tau$.
		\item We set $U_\infty := U$ and $U_0:=U_\infty^\tau$. The other root groups are defined as
		\[U_x:=U_0^{\alpha_x} \ \text{ for $x\nsim\infty$},\qquad U_x:=U_\infty^{\gamma_{x\tau^{-1}}} \ \text{for $x\sim\infty$}.\]
		\item As in the definition of local Moufang sets, we write $U_{\overline{x}}$ for the induced action of $U_x$ on~$\overline{X}$.
	\end{itemize}
	This gives us all the data that is needed for a local Moufang set; we denote the result of this construction by $\M(U,\tau)$.
    We will need some additional definitions, which we have seen before for local Moufang sets, but which we need to redefine in the current setup:
	\begin{itemize}
		\item We call $x\in X$ a unit if $x\nsim0$ and $x\nsim\infty$.
		\item For $x\nsim\infty$, we set $-x:=0\alpha_x^{-1}$.
		\item For a unit $x$, we define the $\mu$-map $\mu_x:=\gamma_{(-x)\tau^{-1}}\alpha_x\gamma_{-(x\tau^{-1})}$.
	\end{itemize}
\end{construction}

This construction does not always give rise to a local Moufang set. The following theorem gives some useful criteria.

\begin{theorem}\label{thm:constrMouf}
	Let $\M(U,\tau)$ be as in Construction~\ref{constr:MUtau}. Then $\M(U,\tau)$ is a local Moufang set if and only if one of the following equivalent conditions holds:
	\begin{enumerate}[itemsep=0.4ex]
		\item $U_\infty^{\gamma_{x\tau^{-1}}} = U_x$ for all units $x\in X$; \label{lem:gleq3}
		\item $U_0^{\mu_x} = U_\infty$ for all units $x\in X$; \label{lem:gleq4}
		\item $U_0 = U_\infty^{\mu_x}$ for all units $x\in X$. \label{lem:gleq5}
	\end{enumerate}
\end{theorem}
\begin{proof}
	This is \cite[Lemma 4.4 and Theorem 4.6]{DMRijckenPSL}.
\end{proof}

We repeat Remark~4.7 from \cite{DMRijckenPSL}:

\begin{remark}\label{remark:constructioniso}
	Let $\M(U,\tau)$ and $\M(U',\tau')$ be given by Construction~\ref{constr:MUtau}, with actions on $(X,\lsim)$ and $(X',\lsim')$ respectively,
    and assume that there is a bijection $\phi \colon X\to X'$ and a group isomorphism $\theta \colon U\to U'$ such that
	\begin{itemize}
		\item for all $x,y\in X$, we have $x\sim y \iff \phi(x)\sim'\phi(y)$;
		\item for all $x\in X$ and $u\in U$, we have $\phi(x\cdot u) = \phi(x)\cdot \theta(u)$;
		\item for all $x\in X$, we have $\phi(x\cdot\tau) = \phi(x)\cdot\tau'$.
	\end{itemize}
    Then $\M(U,\tau)$ and $\M(U',\tau')$ are isomorphic.
\end{remark}

\section{Special local Moufang sets}\label{sec:SpecialLMS}

\subsection{Basic properties}

In \cite{DMRijckenPSL}, we used the notion of special local Moufang sets to characterize the local Moufang sets of the form $\M(R)$ for a local ring $R$. We will need to extend the theory of special local Moufang sets in order to do the same for local Moufang sets originating from local Jordan pairs.
In particular, we will study the $k$-divisibility of elements, similarly to what had been done for (ordinary) Moufang sets in \cite[Proposition~4.6]{MR2425693}.

\begin{definition}
	A local Moufang set $\M$ is called \emph{special} if $\til x=-x$ for all units $x\in X$,
    or equivalently, if $(-x)\tau=-(x\tau)$ for all units $x\in X$.
\end{definition}

\begin{lemma}\label{lem:specialmu}
	Let $x\in X$ be a unit in a special local Moufang set. Then
	\begin{enumerate}[topsep=0pt,itemsep=0.4ex]
		\item $(-y)\mu_x = -(y\mu_x)$ for all units $y\in X$; \label{itm:special-i}
		\item $\mu_x = \alpha_x\alpha_{-x\tau^{-1}}^\tau\alpha_x$; \label{itm:special-ii}
		\item $-x = x\mu_x = x\mu_{-x}$; \label{itm:special-iii}
		\item $\mu_x = \alpha_x\alpha_x^{\mu_{\pm x}}\alpha_x$; \label{itm:special-iv}
		\item $\mu_{-x} = \alpha_x\mu_{-x}\alpha_x\mu_{-x}\alpha_x$. \label{itm:special-v}
	\end{enumerate}
\end{lemma}
\begin{proof}
	This is \cite[Lemma 5.2]{DMRijckenPSL}.
\end{proof}

\begin{lemma}\label{lem:specialsum}
	If $x,y\in X$ are units in a special local Moufang set, and $x\alpha_y$ is a unit, then
	\[x\mu_{x\alpha_y} = (-y)\alpha_{-x}\alpha_{x\mu_y}\alpha_{-y}\,.\]
\end{lemma}
\begin{proof}
	This is \cite[Lemma 5.3]{DMRijckenPSL}.
\end{proof}

\begin{definition}
	For $x\in X\setminus\overline{\infty}$ and $n\geq1$, we define $x\cdot n:=0\cdot\alpha_x^n$.
\end{definition}

\begin{lemma}\label{lem:timesn_sim}
	Let $n\in\N$ and $x,y\in X\setminus\overline{\infty}$. If $x\sim y$, then $x\cdot n\sim y\cdot n$.
\end{lemma}
\begin{proof}
	First observe that by \ref{itm:LM2'}, we have
	\[x\sim y \iff 0\alpha_x\alpha_y^{-1}\sim 0 \iff \overline{0}\overline{\alpha}_x\overline{\alpha}_y^{-1} = \overline{0}
			\iff \overline{\alpha}_x\overline{\alpha}_y^{-1} = \id \iff \overline{\alpha_x} = \overline{\alpha_y}\;.\]
	Hence, we get
	\[x\sim y\iff \overline{\alpha_x} = \overline{\alpha_y}\implies \overline{\alpha_x}^n = \overline{\alpha_y}^n \iff \overline{\alpha_{x\cdot n}} = \overline{\alpha_{y\cdot n}} \iff x\cdot n\sim y\cdot n\;,\]
	so indeed $x\sim y\implies x\cdot n\sim y\cdot n$.
\end{proof}

\begin{lemma}\label{lem:ndiv}
	Let $n\in\N$ and $x$ be a unit in a special local Moufang set. If $x\cdot k$ is a unit for all $k\leq n$, then the following hold:
	\begin{enumerate}[topsep=0pt,itemsep=0.4ex]
		\item $(x\cdot k)\mu_{-x}\cdot k = -x$ for all $k\leq n$.\label{lem:ndiv-i}
		\item $x\tau\cdot k$ is a unit for all $k\leq n$.\label{lem:ndiv-ii}
		\item $(x\cdot k)\tau\cdot k = x\tau$ for all $k\leq n$.\label{lem:ndiv-iii}
		\item For all $k\leq n$, $y_k := (-x\cdot k)\mu_{-x}$ is the unique element such that $y_k\cdot k = x$, and $y_k\cdot \ell$ is a unit for all $\ell\leq n$.\label{lem:ndiv-iv}
		\item $(x\cdot n)\cdot k$ is a unit for all $k\leq n$.\label{lem:ndiv-v}
	\end{enumerate}
\end{lemma}
\begin{proof}
	We prove all these statements simultaneously by induction on $n$. Observe that they clearly hold for $n=1$, using Proposition \ref{lem:specialmu}. We now assume that the lemma holds for $n$ and all $x$ satisfying the conditions, and prove it for $n+1$. Hence, we now assume $x\cdot k$ is a unit for $k\leq n+1$.

	We first claim the following:
	\begin{align}\label{eq:bn(n+1)k-unit}
		y_n\cdot(n+1)\cdot k\text{ is a unit for all $k\leq n$.}
	\end{align}
	Suppose this were not the case; then $y_n\cdot(n+1)\cdot k\sim 0$, so $x\cdot(n+1)\cdot k = y_n\cdot(n+1)\cdot k\cdot n\sim 0$. From this, we get $x\cdot n\cdot k\sim-x\cdot k$, and hence $(x\cdot n\cdot k)\mu_x\cdot k\sim(-x\cdot k)\mu_x\cdot k$. Using \ref{lem:ndiv-v} and \ref{lem:ndiv-iii} of the induction hypothesis, we get
	\[(x\cdot n)\mu_x \sim (-x)\mu_x \text{, so } x\cdot n\sim -x \text{ and hence } x\cdot(n+1)\sim 0\;,\]
	which contradicts the assumption.

	Now we prove \ref{lem:ndiv-i}. By induction, $(x\cdot k)\mu_{-x}\cdot k = -x$ for all $k\leq n$, so we only need to show this for $k=n+1$. We have
	\begin{align*}
		-(x\cdot(n+1))\mu_{-x} &= (-x\cdot(n+1))\mu_{-x} \\
		&= (-x\cdot n)\alpha_{-x}\mu_{-x} \\
		&=(-x\cdot n)\mu_{-x}\alpha_x\mu_{-x}\alpha_x & \text{(by Proposition \ref{lem:specialmu}\ref{itm:special-v})} \\
		&= y_n\alpha_x\mu_{-x}\alpha_x & \text{(by the induction hypothesis)} \\ 
		&= y_n\alpha_{y_n}^n\mu_{-x}\alpha_x \\
		&= (y_n\cdot(n+1))\mu_{-x}\alpha_x \\
		&= \bigl((y_n\cdot(n+1)\cdot n)\mu_{-x}\cdot n\bigr)\alpha_x & \text{(by the induction hypothesis and \eqref{eq:bn(n+1)k-unit})} \\ 
		&= \bigl((x\cdot(n+1))\mu_{-x}\cdot n\bigr)\alpha_x\;.
	\end{align*}
	Hence $\alpha_{(x\cdot(n+1))\mu_{-x}}^{-1} = \alpha_{(x\cdot(n+1))\mu_{-x}}^n\alpha_x$, so indeed
    \[ -x = 0\cdot \alpha_{-x} = 0\cdot\alpha_{(x\cdot(n+1))\mu_{-x}}^{n+1} = (x\cdot(n+1))\mu_{-x}\cdot(n+1)\;. \]
	Next, we prove \ref{lem:ndiv-ii}, where again, we only need to check that $x\tau\cdot(n+1)$ is a unit. By Lemma~\ref{prop:huaAut}, $(x\cdot n)\mu\tau = x\mu\tau\cdot n$ for any $\mu$-maps $\tau$ and $\mu$. Hence
	\[ x\tau\cdot(n+1) = (-x)\mu_{-x}\tau\cdot(n+1) = (-x)\cdot(n+1)\mu_{-x}\tau\nsim 0\;, \]
	as $x\cdot(n+1)\nsim 0$.

	Similarly, we prove \ref{lem:ndiv-iii} using \ref{lem:ndiv-i} and Lemma~\ref{prop:huaAut}.
	\[ (x\cdot k)\tau\cdot k = (x\cdot k)\mu_{-x}\mu_x\tau\cdot k = \bigl((x\cdot k)\mu_{-x}\cdot k\bigr)\mu_x\tau = -x\mu_x\tau = x\tau\;.  \]

	To prove \ref{lem:ndiv-iv}, we first observe that $y_{n+1} = (-x\cdot(n+1))\mu_{-x}$ indeed satisfies $y_{n+1}\cdot(n+1) = x$, by \ref{lem:ndiv-iii}. We first show the following statement.
	\begin{align}\label{eq:ck-unit}
		\text{If $z\cdot(n+1)=x$, then $z\cdot k$ is a unit for all $k\leq n+1$.}
	\end{align}
	Indeed, if $z\cdot k\sim 0$, then also $z\cdot k\cdot(n+1)\sim 0$, so $x\cdot k\sim 0$, contradicting the fact that $x\cdot k$ is a unit for all $k\leq n+1$.

	Now we prove that $y_{n+1}$ is unique. Suppose $z\cdot(n+1)=x$, then
	\begin{align*}
		-x\cdot(n+1) &= x\mu_{-x}\cdot (n+1) = (z\cdot(n+1))\mu_{-x}\cdot (n+1) = z\mu_{-x}\;,
	\end{align*}
	using \ref{lem:ndiv-iii} for $z$, which is allowed by \eqref{eq:ck-unit}. Hence, $z=y_{n+1}$, and indeed $y_{n+1}$ is unique. By~\eqref{eq:ck-unit}, $y_{n+1}\cdot k$ is a unit for $k\leq n+1$. We only need to show that $y_k\cdot(n+1)$ is a unit for $k\leq n$. Suppose $y_k\cdot(n+1)\sim 0$, then also $x\cdot(n+1) = y_k\cdot(n+1)\cdot k\sim 0$, which is a contradiction.

	It only remains to show \ref{lem:ndiv-v}, which we do in two steps. First, if $x\cdot(n+1)\cdot k\sim 0$ for some $k\leq n$, we would have
	\begin{align*}
		& x\cdot n\cdot k \sim -x\cdot k \\
		\implies& (x\cdot n\cdot k)\mu_x\cdot k \sim (-x\cdot k)\mu_x\cdot k \\
		\implies& (x\cdot n)\mu_x \sim (-x)\mu_x &\text{(by the induction hypothesis and \ref{lem:ndiv-iii})}\\
		\implies& x\cdot n \sim -x \\
		\implies& x\cdot(n+1)\sim 0,
	\end{align*}
	which is a contradiction; so $x\cdot(n+1)\cdot k$ is a unit for $k\leq n$.

	Now if $x\cdot(n+1)\cdot(n+1)\sim 0$, we would have
	\begin{align*}
		& -x\cdot(n+1)\cdot n \sim x\cdot(n+1) \\
		\implies& -(x\cdot(n+1)\cdot n)\mu_x\cdot n \sim (x\cdot(n+1))\mu_x\cdot n \\
		\implies& -(x\cdot(n+1))\mu_x \sim (x\cdot(n+1))\mu_x\cdot n &\text{(using the previous step)}\\
		\implies& (x\cdot(n+1))\mu_x\cdot(n+1)\sim 0 \\
		\implies& -x\sim 0 ,
	\end{align*}
	which is again a contradiction. This shows that $x\cdot(n+1)\cdot k$ is a unit for $k\leq n+1$.

	By induction, the lemma now holds for all $n$.
\end{proof}

The statement of the previous lemma is quite technical  in order to make the induction work.
The essence of it is contained in the following corollary.

\begin{corollary}\label{cor:ndiv}
	Let $\M$ be a special local Moufang set and assume $x\cdot k$ is a unit for all $k\leq n$.
	\begin{enumerate}[topsep=0pt]
		\item there is a unique $y$ such that $y\cdot n = x$, which we denote by $x\cdot\frac{1}{n}$;
		\item $(x\cdot n)\tau = x\tau\cdot\frac{1}{n}$ and $\bigl(x\cdot\frac{1}{n}\bigr)\tau = x\tau\cdot n$;
		\item if $z\sim x$, then $z\cdot k$ is a unit for all $k\leq n$ and $x\cdot\frac{1}{n}\sim z\cdot\frac{1}{n}$.
	\end{enumerate}
\end{corollary}
\begin{proof}\leavevmode
	\begin{enumerate}
		\item By Lemma~\ref{lem:ndiv}\ref{lem:ndiv-iv}, $y:=(-x\cdot n)\mu_{-x}$ is the unique element satisfying $y\cdot n=x$.
		\item Lemma~\ref{lem:ndiv}\ref{lem:ndiv-iii} gives us $(x\cdot n)\tau\cdot n = x\tau$, so $(x\cdot n)\tau = x\tau\cdot\frac{1}{n}$. By Lemma~\ref{lem:ndiv}\ref{lem:ndiv-iv}, $x\cdot\frac{1}{n}$ also satisfies the conditions of Lemma~\ref{lem:ndiv}, so we have $\bigl(\bigl(x\cdot\frac{1}{n}\bigr)\cdot n\bigr)\tau\cdot n = \bigl(x\cdot\frac{1}{n}\bigr)\tau$, hence $x\tau\cdot n = \bigl(x\cdot\frac{1}{n}\bigr)\tau$.
		\item By Lemma~\ref{lem:timesn_sim}, $z\cdot k \sim x\cdot k\nsim 0$ for all $k\leq n$. Now we have
		\begin{align*}
			x\sim z &\implies x\tau^{-1}\sim z\tau^{-1}\implies x\tau^{-1}\cdot n\sim z\tau^{-1}\cdot n \\
			&\implies (x\tau^{-1}\cdot n)\tau\sim (z\tau^{-1}\cdot n)\tau\implies x\cdot\tfrac{1}{n}\sim z\cdot\tfrac{1}{n}\qedhere
		\end{align*}
	\end{enumerate}
\end{proof}

\subsection{Special local Moufang sets with abelian root groups}

In this subsection, we will assume we have a special local Moufang set with $U_\infty$ abelian. Since all root groups are conjugate in the little projective group, this means all root groups are abelian.

\begin{proposition}\label{prop:mu-involution}
	Let $x$ and $y$ be units in a special local Moufang set with $U_\infty$ abelian. Then
	\begin{enumerate}
		\item $\mu_x = \mu_{-x} = \mu_x^{-1}$, so $\mu_x^2 = \id$;\label{itm:mu-involution}
		\item $\mu_x^{\mu_y} = \mu_{x\mu_y}$;
		\item if $x\alpha_y$ is a unit, then $\mu_x\mu_{x\alpha_y}\mu_y = \mu_y\mu_{x\alpha_y}\mu_x = \mu_{(x\tau\alpha_{y\tau})\tau}$. \label{itm:mucommute}
	\end{enumerate}
\end{proposition}
\begin{proof}\leavevmode
	\begin{enumerate}
		\item This is \cite[Lemma 5.8]{DMRijckenPSL}.
		\item By Lemma~\ref{prop:mu}\ref{itm:mutau}, we have $\mu_{x\mu_y} = \mu_{-x}^{\mu_y}$, so this follows from~\ref{itm:mu-involution}.
		\item Let $z = x\tau\alpha_{y\tau}\tau$. Then, by Proposition~\ref{prop:sumform}, we have $\mu_{-y}\mu_z\mu_{-x} = \mu_{(-x)\alpha_{-y}}$, so by \ref{itm:mu-involution} we get $\mu_z = \mu_x\mu_{y\alpha_x}\mu_y$. If we interchange $x$ and $y$, $z$ remains the same by the commutativity of $U_{\infty}$, so we also get $\mu_z = \mu_y\mu_{x\alpha_y}\mu_x$. By the commutativity of $U_{\infty}$ again, $x\alpha_y = 0\alpha_x\alpha_y = y\alpha_x$.\qedhere
	\end{enumerate}
\end{proof}

When a special local Moufang set has abelian root groups, we can extend Corollary~\ref{cor:ndiv} to non-units in the sense that the root groups will be uniquely $n$-divisible.

\begin{definition}
	A group $U$ is uniquely $k$-divisible if for every $g\in U$ there is a unique $h\in U$ such that $h^k=g$ (or such that $h\cdot k = g$ if we write the group operation additively). We denote $h$ as $g/k$, $g \cdot \frac{1}{k}$ or $g \cdot k^{-1}$.
\end{definition}

\begin{proposition}\label{prop:ndivglobal}
	Let $\M$ be a special local Moufang set with $U_\infty$ abelian, and $n\in\N$ a natural number. If for all units $x$ and all $k\leq n$, $x\cdot k$ is also a unit, then $U_\infty$ is uniquely $k$-divisible for all $k\leq n$.
\end{proposition}
\begin{proof}
	Let $k\leq n$. Corollary~\ref{cor:ndiv} already shows that, if $x$ is a unit, there is a unique $y$ such that $y\cdot k=x$; therefore, it only remains to check the unique $k$-divisibility for non-units.
	So suppose that $x$ is not a unit. Take any unit $e$; then $\alpha_x = \alpha_{x\alpha_{-e}}\alpha_{e}$. Now $x\alpha_{-e}$ and $e$ are units, so both $\alpha_{x\alpha_{-e}}$ and $\alpha_e$ are uniquely $k$-divisible, say with $y\cdot k = x\alpha_{-e}$ and $z\cdot k = \alpha_e$. Since $U_\infty$ is abelian, we get
	\[(\alpha_y\alpha_z)^k = \alpha_y^k\alpha_z^k = \alpha_{x\alpha_{-e}}\alpha_{e} = \alpha_x\,.\]
	To show uniqueness, suppose there are two elements $u,u' \in U_\infty$ with $u^k = u'^k = \alpha_x$. Then
	\[(\alpha_y^{-1}u)^k = \alpha_y^{-k}\alpha_x = \alpha_y^{-k}\alpha_y^k\alpha_z^k = \alpha_e\,,\]
	and similarly $(\alpha_y^{-1}u')^k = \alpha_e$. By the uniqueness for units, we get $\alpha_y^{-1}u' = \alpha_y^{-1}u$, so $u=u'$.
\end{proof}

\begin{proposition}\label{prop:mu-x.s-part}
	Let $\M$ be a special local Moufang set with $U_\infty$ abelian, and $n\in\N$ a natural number. Assume that for all units $x$ and all $k\leq n$, $x\cdot k$ is also a unit. Then we have $y\mu_{x\cdot\ell} = y\mu_x\cdot\ell^2$ for all units $x,y$ and for $\ell\in\{n,n^{-1}\}$.
\end{proposition}
\begin{proof}
	Let $y$ be a unit. Then
	\begin{align*}
		y\mu_{x\cdot n} &= y\alpha_{x\cdot n}\tau\alpha_{-(x\cdot n)\tau}\tau\alpha_{x\cdot n} \\
		&= \bigl(\bigl(y\cdot\tfrac{1}{n}\bigr)\alpha_x\cdot n\bigr)\tau\alpha_{-(x\cdot n)\tau}\tau\alpha_{x\cdot n} \\
		&= \bigl(\bigl(y\cdot\tfrac{1}{n}\bigr)\alpha_x\tau\cdot\tfrac{1}{n}\bigr)\alpha_{-x\tau\cdot\tfrac{1}{n}}\tau\alpha_{x\cdot n} \\
		&= \bigl(\bigl(y\cdot\tfrac{1}{n}\bigr)\alpha_x\tau\alpha_{-x\tau}\cdot\tfrac{1}{n}\bigr)\tau\alpha_{x\cdot n} \\
		&= \bigl(\bigl(y\cdot\tfrac{1}{n}\bigr)\alpha_x\tau\alpha_{-x\tau}\tau\cdot n\bigr)\alpha_{x\cdot n} \\
		&= \bigl(y\cdot\tfrac{1}{n}\bigr)\alpha_x\tau\alpha_{-x\tau}\tau\alpha_{x}\cdot n \\
		&= y\cdot\tfrac{1}{n}\mu_x\cdot n = x\mu_x\cdot n^2\;.
	\end{align*}
	Substituting $x$ by $x\cdot \tfrac{1}{n}$, we get
	\[y\mu_x = y\mu_{x\cdot\tfrac{1}{n}}\cdot n^2\text{, so }y\mu_{x\cdot\tfrac{1}{n}} = y\mu_x\cdot\tfrac{1}{n^2}\;,\]
	hence $y\mu_{x\cdot\ell} = y\mu_x\cdot\ell^2$ for both values of $\ell$.
\end{proof}

We would now like to know what the $y\mu_{x\cdot\ell}$ is when $y$ is not a unit. Of course, if $y\sim 0$, we get $y\mu_x \sim \infty$, so it would not make sense to compare $y\mu_{x\cdot\ell}$ to $y\mu_x \cdot\ell^2$, since the second expression does not make sense. To resolve this, we also use the `multiplication by $n$' for $U_0$.

\begin{definition}
	For $x\in X\setminus\overline{0}$ and $n\geq1$, we define $x\cdottil n:=\infty\cdot\gamma_{x\tau^{-1}}^n$.
\end{definition}
\begin{remark}
	Even though $\tau$ appears in $\gamma_{x\tau^{-1}}$, remember that this is the unique element of $U_0$ mapping $\infty$ to $x$, which, therefore, does not depend on the choice of $\tau$. Hence $\cdottil n$ is also independent of this choice.
\end{remark}

Since this is exactly what we get when we switch the role of $0$ and $\infty$, we immediately know that Corollary~\ref{cor:ndiv} and Proposition~\ref{prop:ndivglobal} also hold for $\cdottil$. The first thing we can observe is that there is a close relation between $\cdot$ and $\cdottil$ :

\begin{lemma}\leavevmode
	\begin{enumerate}
		\item If $x\nsim\infty$, then $(x\cdot n)\tau = x\tau\cdottil n$. \label{itm:xnt}
		\item If $x\nsim0$, then $(x\cdottil n)\tau = x\tau\cdot n$.
	\end{enumerate}
\end{lemma}
\begin{proof}\leavevmode
	\begin{enumerate}
		\item We have
		\[x\tau\cdottil n = \infty\gamma_x^n = \infty\alpha_x^{\tau n} = \infty\tau^{-1}\alpha_x^n\tau = (0\alpha_x^n)\tau = (x\cdot n)\tau\;.\]
		\item This follows from~\ref{itm:xnt}, using $x\tau^{-1}$ and replacing $\tau^{-1}$ by $\tau$.\qedhere
	\end{enumerate}
\end{proof}

Combining this with Corollary~\ref{cor:ndiv}, we are able to express $\cdot\frac{1}{n}$ in terms of $\cdottil$ and we can extend Proposition~\ref{prop:mu-x.s-part}:
\begin{proposition}\label{prop:mu-x.s}
	Let $\M$ be a special local Moufang set with abelian root groups, and $n\in\N$ a natural number. Assume that for all units $x$ and all $k\leq n$, $x\cdot k$ is also a unit. Let $\ell\in\{n,n^{-1}\}$. Then for all units $x$, we have
	\begin{enumerate}
		\item $x\cdottil\ell = x\cdot\ell^{-1}$, hence $(x\cdot\ell)\tau = x\tau\cdottil\ell$ and $(x\cdottil\ell)\tau = x\tau\cdot\ell$;
		\item For $y\nsim\infty$, we have $y\mu_{x\cdot\ell^{-1}} = y\mu_{x\cdottil\ell} = y\mu_x\cdottil\ell^{2}$;\label{itm:mu-x.sfirst}
		\item For $y\nsim 0$, we have $y\mu_{x\cdot\ell} = y\mu_x\cdot\ell^2$.\label{itm:mu-x.s}
	\end{enumerate}
\end{proposition}
\begin{proof}\leavevmode
	\begin{enumerate}
		\item By Corollary~\ref{cor:ndiv} and the previous lemma, we have
		\[(x\cdottil n)\tau = x\tau\cdot n = \bigl(x\cdot\frac{1}{n}\bigr)\tau\;.\]
		so $x\cdottil n = x\cdot\frac{1}{n}$. By switching the roles of $0$ and $\infty$, we also get $x\cdottil\frac{1}{n} = x\cdot n$. Furthermore, we get $x\tau\cdottil\ell = x\tau\cdot\ell^{-1} = (x\cdot\ell)\tau$ and $x\tau\cdot\ell = (x\cdot\ell^{-1})\tau = (x\cdottil\ell)\tau$.
		\item By applying $\tau$ to both sides of the identity we want to prove, we get the equivalent identity $y\mu_{x\cdot\ell^{-1}}\tau = y\mu_x\tau\cdot\ell^2$. By Proposition~\ref{prop:mu-x.s-part}, this identity holds if $y$ is a unit. So assume now that $y\sim 0$. Let $e$ be a unit, then $y = y\alpha_{-e}\alpha_e$, so using Lemma~\ref{prop:huaAut} we get
		\begin{align*}
			y\mu_{x\cdot\ell^{-1}}\tau &= y\alpha_{-e}\alpha_e\mu_{x\cdot\ell^{-1}}\tau \\
				&= (y\alpha_{-e})\mu_{x\cdot\ell^{-1}}\tau\alpha_e^{\mu_{x\cdot\ell^{-1}}\tau} \\
				&= (y\alpha_{-e}\mu_x\tau\cdot\ell^2)\alpha_{e\mu_{x\cdot\ell^{-1}}\tau} \\
				&= (y\alpha_{-e}\mu_x\tau\cdot\ell^2)\alpha_{e\mu_x\tau\cdot\ell^2} \\
				&= (y\alpha_{-e}\mu_x\tau\alpha_{e\mu_x\tau})\cdot\ell^2 \\
				&= (y\alpha_{-e}\alpha_e)\mu_x\tau\cdot\ell^2 \\
				&= y\mu_x\tau\cdot\ell^2\;.
		\end{align*}
		Hence the desired identity also holds for $y\sim 0$.
		\item This is precisely the \ref{itm:mu-x.sfirst} with $0$ and $\infty$ interchanged.\qedhere
	\end{enumerate}
\end{proof}

\section{From local Jordan pairs to local Moufang sets}\label{sec:JP}

\subsection{Preliminaries on local Jordan pairs}

Throughout this section, we will be working with Jordan pairs. First we recall some notations and definitions from \cite{MR0444721}. Remark that we will change the left action of {\em loc.\@ cit.} to a right action, in order to be consistent with the action of our local Moufang sets. The index $\sigma$ will always take values $+$ and $-$.

\begin{definition}
	Let $k$ be a commutative unital ring and $V = (V^+,V^-)$ a pair of $k$-modules with quadratic maps $Q\colon V^\sigma\to \Hom(V^{-\sigma},V^\sigma)$.
	We write $Q_{x,z} := Q_{x+z}-Q_x-Q_z$, $zD_{x,y} := yQ_{x,z}$ and $\{xyz\}:=yQ_{x,z}$.
	Then $V$ is a \emph{Jordan pair} if the following axioms are satisfied in all scalar extensions of the base ring:
			\begin{enumerate}[label=\textnormal{(JP\arabic*)},leftmargin=*]
				\item $\{x\,y\,zQ_x\} = \{yxz\}Q_x$;\label{axiom:JP1}
				\item $\{yQ_x\,y\,z\} = \{x\,xQ_y\,z\}$;\label{axiom:JP2}
				\item $Q_{yQ_x} = Q_xQ_yQ_x$.\label{axiom:JP3}
			\end{enumerate}
	A pair of submodules $U = (U^+,U^-)$ is an \emph{ideal} if $vQ_u\in U^\sigma$, $uQ_v\in U^{-\sigma}$ and $\{v'vu\}\in U^\sigma$ for all $u\in U^\sigma, v\in V^{-\sigma}, v'\in V^\sigma$. If $(U^+,U^-)$ is an ideal, the quotient $V/U = (V^+/U^+, V^-/U^-)$ is a Jordan pair. An ideal $U$ is \emph{proper} if $U\neq V$.
    A \emph{homomorphism} of Jordan pairs is a pair of $k$-linear maps $h_\sigma\colon V^\sigma\to W^\sigma$ such that
	\[h_\sigma(yQ_x) = h_{-\sigma}(y)Q_{h_\sigma(x)}\quad\text{for all $x\in V^\sigma,y\in V^{-\sigma}$.}\]
\end{definition}

The following proposition gives some useful criteria to check whether a given structure is a Jordan pair.
\begin{proposition}\leavevmode\label{prop:JPsufficientaxioms}
	\begin{enumerate}
		\item If $V$ has no $2$-torsion, then \ref{axiom:JP3} follows from \hyperref[axiom:JP1]{\normalfont{(JP1-2)}}. Hence in this case $V$ is a Jordan pair if \hyperref[axiom:JP1]{\normalfont{(JP1-2)}} are satisfied in all scalar extensions of the base ring.
		\item If $V$ has no $2$-torsion and \ref{axiom:JP1}, all its linearizations and \ref{axiom:JP2} hold, then $V$ is a Jordan pair.\label{itm:sufficient_JP}
	\end{enumerate}
\end{proposition}
\begin{proof}\leavevmode
	\begin{enumerate}
		\item This is \cite[Proposition 2.2(a)]{MR0444721}.
		\item This is remarked just after \cite[Definition 1.2]{MR0444721}.\qedhere
	\end{enumerate}
\end{proof}

We will also need the notions of invertibility, (properly) quasi-invertibility, and of course of a local Jordan pair, again from \cite{MR0444721}.

\begin{definition}
	An element $v\in V^\sigma$ is \emph{invertible} if and only if $Q_v$ is invertible. In this case, we define $v^{-1}:=vQ_v^{-1}$. A Jordan pair is \emph{division} if all non-zero elements are invertible. A Jordan pair is \emph{local} if the non-invertible elements form a proper ideal.
	For $(x,y)\in V$ (this means $x\in V^+, y\in V^-$), we define the \emph{Bergman operator}
    \[ B_{x,y} := \id - D_{x,y} + Q_yQ_x \;, \]
    and $(x,y)$ is \emph{quasi-invertible} if and only if $B_{x,y}$ is invertible. In this case, we define the \emph{quasi-inverse} $x^y := (x-yQ_x)B^{-1}_{x,y}$. An element $x\in V^+$ (or $y\in V^-$) is \emph{properly quasi-invertible} if and only if $(x,y)$ is quasi-invertible for all $y\in V^-$ (or all $x\in V^+$, respectively). The \emph{Jacobson radical} $\Rad V = (\Rad V^+, \Rad V^-)$ is the pair of sets of all properly quasi-invertible elements.
\end{definition}

To do computations in local Jordan pairs, we will need some further properties and identities:

\begin{proposition}\leavevmode\label{prop:JPbasic}
	\begin{enumerate}
		\item For any $x$ and $y$, we have, $Q_{x,yQ_x} = Q_xD_{x,y} = D_{y,x}Q_x$.
		\item For invertible $x$ and any $y$, we have $Q_{x,y}Q_x^{-1} = D_{x^{-1},y}$. \label{itm:JPbasicid}
		\item If $x$ is invertible, $B_{x,y} = Q_{x^{-1}-y}Q_x$. If $y$ is invertible, we have $B_{x,y} = Q_yQ_{x-y^{-1}}$.\label{itm:JPbasicA}
		\item Assume $(x,y)$ is quasi-invertible. Then $(x,y+z)$ is quasi-invertible if and only if $(x^y,z)$ is quasi-invertible. In this case, we have $x^{y+z} = (x^y)^z$.\label{itm:JPbasicB}
		\item $(x,y)$ is quasi-invertible if and only $(y,x)$ is quasi-invertible in $(V^-,V^+)$. In this case, $x^y = x+y^xQ_x$.\label{itm:JPbasicswitch}
		\item $(x,zQ_y)$ is quasi-invertible if and only if $(xQ_y,z)$ is quasi-invertible. In this case, $(xQ_y)^z = x^{zQ_y}Q_y$.\label{itm:JPbasicQy}
		\item If $V$ is a local Jordan pair, then $\Rad V$ is the set of non-invertible elements of $V$.\label{itm:JPbasicD}
		\item If $V/\Rad V$ is a non-trivial Jordan division pair, then $V$ is a local Jordan pair.
		\item If $(x,y)\mod \Rad V$ is quasi-invertible in the quotient $V/\Rad V$, then also $(x,y)$ is quasi-invertible.\label{itm:JPbasicQIlift}
		\item If $x\in \Rad V^+$ and $y\in V^-$, then $x^y\in\Rad V^+$.
		\item If $x,y\in V^+$ are invertible and $x-y\in\Rad V^+$, then $x^{-1}-y^{-1}\in\Rad V^-$.\label{itm:JPbasicInvertRadical}
	\end{enumerate}
\end{proposition}
\begin{proof}\leavevmode
	\begin{enumerate}
		\item This is JP4 in \cite[2.1]{MR0444721}.
		\item From the definition of $Q_{y,z}$, it is clear that $Q_xQ_{y,z}Q_x = Q_{yQ_x,zQ_x}$, so
		\[Q_{x,y}Q_x^{-1} = Q_xQ_x^{-1}Q_{x,y}Q_x^{-1} = Q_xQ_{x^{-1},yQ_x^{-1}} = Q_xQ_x^{-1}D_{x^{-1},y} = D_{x^{-1},y}\;.\]
		\item This is \cite[2.12]{MR0444721}.
		\item This is \cite[3.7(1)]{MR0444721}.
		\item This is \cite[3.3]{MR0444721}.
		\item This is \cite[3.5(1)]{MR0444721}.
		\item This is \cite[4.4(a)]{MR0444721}.
		\item This is \cite[4.4(b)]{MR0444721}.
		\item This is \cite[4.3]{MR0444721}.
		\item If $x\in\Rad V^+$, then $(x,z)$ is quasi-invertible for all $z\in V^-$. Hence $(x,y+z)$ is quasi-invertible for all $z\in V^-$, so $(x^y,z)$ is quasi-invertible for all $z\in V^-$. Hence $x^y$ is properly quasi-invertible, and $x^y\in\Rad V^+$.
		\item Since $\Rad V$ is an ideal and $Q_x$ is invertible, it is sufficient to prove that $(x^{-1}-y^{-1})Q_x$ is in $\Rad V^+$. This is indeed true, as
		\begin{align*}
			(x^{-1}-y^{-1})Q_x &= x - y^{-1}Q_{x-y+y} = x - y^{-1}Q_{x-y,y} - y^{-1}Q_{x-y} - y^{-1}Q_{y} \\
				&= (x-y) - y^{-1}Q_{x-y,y} - y^{-1}Q_{x-y} \in \Rad V^+\;,
		\end{align*}
		since $x-y\in\Rad V^+$.\qedhere
	\end{enumerate}
\end{proof}

We recall the connection between (local) Jordan algebras and (local) Jordan pairs.
\begin{proposition}
    Let $J$ be a quadratic Jordan algebra with quadratic maps $U\colon J\to\End(J)$. Then:
	\begin{enumerate}
        \item $(J,J)$ is a Jordan pair with $Q := U$.\label{itm:alg-pair}
		\item $J$ is a Jordan division algebra if and only if $(J,J)$ is a Jordan division pair.\label{itm:pair-alg-div}
		\item $J$ is a local Jordan algebra if and only if $(J,J)$ is a local Jordan pair.\label{itm:pair-alg-loc}
		\item The radical of the Jordan pair $(J,J)$ is $(\Rad J,\Rad J)$, where $\Rad J$ is the radical of the Jordan algebra $J$.\label{itm:pair-alg-rad}
		\item The map $J\mapsto(J,J)$ induces a bijection from isotopism classes of local Jordan algebras to isomorphism classes of local Jordan pairs.\label{itm:pair-alg-classes}
	\end{enumerate}
\end{proposition}
\begin{proof} The first statement~\ref{itm:alg-pair} is \cite[1.6]{MR0444721}. Statements \ref{itm:pair-alg-div} and \ref{itm:pair-alg-loc} are in \cite[1.10]{MR0444721}, and \ref{itm:pair-alg-classes} is a consequence of \ref{itm:pair-alg-loc} and \cite[1.12]{MR0444721}. Finally, \ref{itm:pair-alg-rad} is one of the statements of \cite[4.17]{MR0444721}.
\end{proof}

\begin{examples}\leavevmode \label{ex:local}
	\begin{enumerate}[label=\textnormal{(\arabic*)},itemsep=0.3ex]
		\item Let $A$ be a local associative (not necessarily commutative) ring. Then $V = (A,A)$ with $Q_a\colon A\to A:x\mapsto axa$ is a local Jordan pair.
		\item Let $V=(V^+,V^-)$ be a Jordan division pair over a field $k$ and let $R$ be a commutative $k$-algebra which is a local ring. Then we can define $V\otimes_k R = (V^+\otimes_k R,V^-\otimes_k R)$ with $(x\otimes r)Q_{y\otimes s}:= xQ_y\otimes rs^2$. This is a local Jordan pair.
		\item Let $J$ be a finite dimensional Jordan division algebra over a field $K$ which is complete with respect to a discrete valuation $v$. Then we can extend the valuation on $K$ to a valuation $v_J$ on $J$. The subalgebra $J_0 = \{x\in J\mid v_J(x)\geq0\}$ is now a local Jordan algebra with $\Rad J_0 = \{x\in J\mid v_J(x)>0\}$ (and hence $(J_0,J_0)$ is a local Jordan pair). We refer to \cite{KnebuschJordanValuation,PeterssonJordanValuations,PeterssonExceptionalJordanValuation} for more details.
	\end{enumerate}
\end{examples}

We will rely on the notion of the projective space of $V$, which was introduced by O. Loos in \cite{MR516835}. The description we use comes from Loos' more recent article \cite{MR1280101}.

\begin{definition}
	Two pairs $(x,y),(x',y')\in V$ are \emph{projectively equivalent} if
	\[(x,y-y')\text{ is quasi-invertible and }x' = x^{y-y'}\;.\]
	Using Proposition~\ref{prop:JPbasic}\ref{itm:JPbasicB}, this can be shown to be an equivalence relation, and we will denote the equivalence class of $(x,y)$ by $[x,y]$. The \emph{projective space of $V$} is the set
	\[\P(V) := \{[x,y]\mid (x,y)\in V\}\;.\]
\end{definition}

\subsection{Defining a local Moufang set \texorpdfstring{$\M(V)$}{M(V)} from a local Jordan pair \texorpdfstring{$V$}{V}}

To define a local Moufang set from a local Jordan pair, we first need a set with equivalence relation. The set we will use is the projective space $\P(V)$ of $V$. In order to define an equivalence relation on $\P(V)$, it will be convenient to have a nice set of representatives for its elements. This description will rely on the choice of an invertible element $e\in V^+$.

\begin{proposition}\label{prop:PVrepr}
	Let $V$ be a local Jordan pair and $e\in V^+$ invertible. For any $(x,y)\in V$, at least one of the following occurs:
	\begin{enumerate}[label*=\textnormal{\Roman*.}]
		\item There is a unique $t\in V^+$ such that $[x,y]=[t,0]$.
		\item There is a unique $t\in V^-$ such that $[x,y]=[e,e^{-1}+t]$.
	\end{enumerate}
	If in either of the cases $t$ is non-invertible, then the other case cannot occur. If $t$ is invertible, we have
	\[[t,0] = [e,e^{-1}-t^{-1}].\]
\end{proposition}
\begin{proof}
	Let $(x,y)\in V$. Assume first that $(x,y)$ is quasi-invertible. Then we immediately have $[x,y] = [x^y,0]$, so we are in the first case.

    So assume now that $(x,y)$ is not quasi-invertible. By Proposition~\ref{prop:JPbasic}\ref{itm:JPbasicD}, this means that $x$ is invertible. In this case, set $t = y-x^{-1}$. Now, using Proposition~\ref{prop:JPbasic}\ref{itm:JPbasicA}, we have
	\begin{align*}
		[x,y] = [e,e^{-1}+t] &\iff (e,e^{-1}-x^{-1})\text{ is quasi-invertible and }e^{e^{-1}-x^{-1}} = x \\
		&\iff B_{e,e^{-1}-x^{-1}}\text{ is invertible and }e-(e^{-1}-x^{-1})Q_e= xB_{e,e^{-1}-x^{-1}} \\
		&\iff Q_{e^{-1}-(e^{-1}-x^{-1})}Q_e\text{ is invertible and }x^{-1}Q_e = xQ_{e^{-1}-(e^{-1}-x^{-1})}Q_e \\
		&\iff Q_{x^{-1}}Q_e\text{ is invertible and }x^{-1}Q_e = xQ_{x^{-1}}Q_e
	\end{align*}
	Now $e$ and $x$ are invertible, so $Q_e$ and $Q_{x^{-1}}$ are invertible, and $x^{-1} = xQ_{x^{-1}}$. So indeed, we found a representative for $[x,y]$ of the second form.

	Now assume that $[t,0]=[e,e^{-1}+s]$ for some $s \in V^-$. Then $B_{e,e^{-1}-s}$ must be invertible, but $B_{e,e^{-1}-s} = Q_{s}Q_e$, since $e$ is invertible. Hence $Q_{s}$ must be invertible, so $s$ must be invertible. In this case
	\[t = e^{e^{-1}+s} = (e - (e^{-1}+s)Q_e)Q_e^{-1}Q_{-s}^{-1} = (e - e-sQ_e)Q_e^{-1}Q_{s}^{-1} = -sQ_{s}^{-1} = -s^{-1}\;,\]
	so $t$ is also invertible. This proves the remaining statements.
\end{proof}

By Proposition \ref{prop:PVrepr}, we now have a nice set of representatives for $\P(V)$ as follows:
\begin{equation}\label{eq:P(V)}
    \P(V) = \{[x,0]\mid x\in V^+\}\cup \{[e,e^{-1}+y]\mid y\in \Rad V^-\}\;.
\end{equation}
The second subset consists of projective points that are ``close'' to each other, in the sense that they only differ by a non-invertible element. We can define a similar closeness relation on the first subset.

\begin{definition}\label{def:radequiv}
	We define a \emph{radical equivalence} relation $\sim$ on $\P(V)$ by
	\begin{align*}
		[x,0]&\sim[x',0]\iff x-x'\in\Rad V^+&&\text{for all $x,x'\in V^+$;} \\
		[e,e^{-1}+y]&\sim [e,e^{-1}+y']\iff y-y'\in\Rad V^-&&\text{for all $y,y'\in V^-$;} \\
		[x,0]&\not\sim [e,e^{-1}+y]&&\text{if $x\in\Rad V^+$ or $y\in\Rad V^-$.}
	\end{align*}
\end{definition}

Observe that this equivalence is well-defined by Proposition~\ref{prop:JPbasic}\ref{itm:JPbasicInvertRadical}.
\begin{remark}
	We could have avoided the explicit choice of representatives for $\P(V)$ by defining the radical equivalence by
    \[ [x,y] \sim [x',y'] \iff
	\begin{aligned}
		&\text{ there are }(\hat{x},\hat{y})\in[x,y]\text{ and }(\hat{x}',\hat{y}')\in[x',y'] \\
		&\text{ such that } (\hat{x},\hat{y})\equiv(\hat{x}',\hat{y}')\mod\Rad V\;.
	\end{aligned}
    \]
\end{remark}
\begin{remark}
	Observe that $[0,0]\nsim[e,e^{-1}]\nsim [e,0]\nsim[0,0]$, so the set of equivalence classes $\overline{\P(V)}$ contains at least $3$ classes.
\end{remark}

We are now prepared to define a local Moufang set corresponding to $V$ using Construction~\ref{constr:MUtau}. We first define what the elements of $U_\infty$ and $U_0$ will be.
\begin{definition}\label{def:alphazetaJP}For all $v\in V^+$:
	\[\alpha_v : \begin{cases}
		[x,0]\mapsto [x+v,0] & \text{for all $x\in V^+$} \\
		[e,e^{-1}+y]\mapsto [e,e^{-1}+y^v] & \text{for all $y\in \Rad V^-$}
	\end{cases}\]
	For all $w\in V^-$:
	\[\zeta_w : \begin{cases}
		[x,0]\mapsto [x^w,0] & \text{for all $x\in \Rad V^+$} \\
		[e,e^{-1}+y]\mapsto [e,e^{-1}+y+w] & \text{for all $y\in V^-$}
	\end{cases}\]
\end{definition}

\begin{proposition}
	The maps $\alpha_v$ and $\zeta_w$ preserve the radical equivalence on $\P(V)$.
\end{proposition}
\begin{proof}
	First, $[x,0]\sim[x',0]$ if and only if $x-x'\in\Rad V^+$, which is equivalent to $(x+v)-(x'+v)\in\Rad V^+$, so in this case $\alpha_v$ preserves equivalence. If furthermore $x\in\Rad V^+$ then $[x,0]\sim[x',0]$ if and only if $x'\in\Rad V^+$. Since $x'\in\Rad V^+\iff x'^w\in\Rad V^+$ and $x^w\in\Rad V^+$, we find that $\zeta_w$ also preserves equivalence.

	Similarly, $[e,e^{-1}+y]\sim[e,e^{-1}+y']$ is equivalent to $[e,e^{-1}+y]\zeta_w\sim[e,e^{-1}+y']\zeta_w$ and if $y\in\Rad V^-$, $[e,e^{-1}+y]\sim[e,e^{-1}+y']\iff[e,e^{-1}+y]\alpha_v\sim[e,e^{-1}+y']\alpha_v$.

	Finally, assume $x\in\Rad V^+$ and $y\in\Rad V^-$, so $[x,0]\nsim[e,e^{-1}+y]$. Then $x^w\in\Rad V^+$, so also $[x,0]\zeta_w\nsim[e,e^{-1}+y]\zeta_w$ and $y^v\in\Rad V^-$, so also $[x,0]\alpha_v\nsim[e,e^{-1}+y]\alpha_v$. These cover all cases.
\end{proof}

We will use the set of all $\alpha_v$ to get $U_\infty$, so it only remains to construct $\tau$ in order to have all the ingredients for a local Moufang set. The following proposition describes the action of what will be the $\mu$-maps of the local Moufang set.
The bulk of the computational work of this section is contained in the proof of this proposition.

\begin{proposition}\label{prop:JordanMuaction}
	Let $v\in V^+$ be invertible and set $\mu_{v} = \zeta_{v^{-1}}\alpha_v\zeta_{v^{-1}}$. Then
	\[\begin{cases}
		[e,e^{-1}+y]\mu_{v} = [yQ_v,0] & \text{for $y\in \Rad V^-$}\;, \\
		[e,e^{-1}+y]\mu_{v} = [e,e^{-1}-y^{-1}Q_v^{-1}] & \text{for $y\in V^-\setminus\Rad V^-$}\;, \\
		[x,0]\mu_{v} = [e,e^{-1}+xQ_v^{-1}] & \text{for $x\in \Rad V^+$}\;.
	\end{cases} \]
	As a consequence, $\mu_v^2 = \id$. Using the other representations, we get
    \begin{alignat*}{2}
        &[e,e^{-1}+y]\mu_{v} = [yQ_v,0] && \text{ for all $y\in V^-$,} \\
        &[x,0]\mu_{v} = [e,e^{-1}+xQ_v^{-1}] \quad && \text{ for all $x\in V^+$ and} \\
        &[x,0]\mu_v = [-x^{-1}Q_v,0] && \text{ for all $x\in V^+\setminus\Rad V^+$.}
    \end{alignat*}
\end{proposition}
\begin{proof}
For simplicity, we will set $w = v^{-1}$ throughout this proof, so $\mu_v = \zeta_w\alpha_v\zeta_w$, $wQ_v = v$ and $vQ_w=w$.

In the first case, we start with $[e,e^{-1}+y]$ for $y\in \Rad V^-$. We get
\begin{align*}
	[e,e^{-1}+y]\mu_v &= [e,e^{-1}+y]\zeta_w\alpha_v\zeta_w = [e,e^{-1}+y+w]\alpha_v\zeta_w = [-(y+w)^{-1},0]\alpha_v\zeta_w \\
		&= [-(y+w)^{-1}+v,0]\zeta_w = [(-(y+w)^{-1}+v)^w,0]\;.
\end{align*}
Hence, we need to check that $y+w$ is invertible and $-(y+w)^{-1}+v\in\Rad V^+$, and we want to prove $(-(y+w)^{-1}+v)^w = yQ_v$. First, since $y\in\Rad V^-$ and $w$ is invertible, clearly $y+w$ is invertible. Second, take $z\in V^-$ arbitrary, then
\[\bigl(-(y+w)^{-1}+v,z\bigr) \equiv \bigl(-w^{-1}+v,z\bigr)\equiv (0,z)\mod \Rad V\;,\]
so $(-(y+w)^{-1}+v,z)\mod\Rad V$ is quasi-invertible, and by Proposition~\ref{prop:JPbasic}\ref{itm:JPbasicQIlift} that means $(-(y+w)^{-1}+v,z)$ is quasi-invertible. As $z$ was arbitrary, that means $-(y+w)^{-1}+v\in \Rad V^+$. Finally
\begin{align*}
	& (v-(y+w)^{-1})^w = yQ_v \\
\iff	& w^{\bigl(v-(y+w)^{-1}\bigr)} = y + w &\text{(by \ref{prop:JPbasic}\ref{itm:JPbasicswitch})} \\
\iff	& \bigl(w-(v-(y+w)^{-1})Q_w\bigr)B_{w,v-(y+w)^{-1}}^{-1} = y+w \\
\iff	& \bigl(w-vQ_w+(y+w)^{-1}Q_w)\bigr)\bigl(Q_{v-(v-(y+w)^{-1})}Q_w\bigr)^{-1} = y+w &\text{(by \ref{prop:JPbasic}\ref{itm:JPbasicA})} \\
\iff	& (y+w)^{-1}Q_wQ_w^{-1}Q_{(y+w)^{-1}}^{-1} = y+w \\
\iff	& (y+w)^{-1}Q_{y+w} = y+w\;,
\end{align*}
which holds, so the identity holds.

In the second case, we start with $[e,e^{-1}+y]$ for $y\in V^-\setminus\Rad V^-$. By $\zeta_w$, this is mapped to $[e,e^{-1}+y+w]$.
We now distinguish two cases according to whether $y+w\in\Rad V^-$ or not.
Assume first that $y+w\in\Rad V^-$; then
\begin{align*}
	[e,e^{-1}+y]\mu_v &= [e,e^{-1}+y+w]\alpha_v\zeta_w = [e,e^{-1}+(y+w)^v]\zeta_w = [e,e^{-1}+(y+w)^v+w]\;.
\end{align*}
We need to check that $(y+w)^v+w = -y^{-1}Q_v^{-1}$:
\begin{align*}
	& (y+w)^v+w = -y^{-1}Q_v^{-1} \\
\iff	& (v^{y+w}-v)Q_v^{-1}+w = -y^{-1}Q_v^{-1} &\text{(by \ref{prop:JPbasic}\ref{itm:JPbasicswitch})} \\
\iff	& v^{y+w}-v+wQ_v = -y^{-1} \\
\iff	& \bigl(v-(y+w)Q_v\bigr)B_{v,y+w}^{-1} = -y^{-1}  \\
\iff	& \bigl(v-(y+w)Q_v\bigr)(Q_{w-(y+w)}Q_v)^{-1} = -y^{-1} &\text{(by \ref{prop:JPbasic}\ref{itm:JPbasicA})} \\
\iff	& \bigl(vQ_v^{-1}-(y+w)\bigr)Q_{-y}^{-1} = -y^{-1} \\
\iff	& -yQ_y^{-1} = -y^{-1}\;,
\end{align*}
which holds, so the identity holds.

Assume now that $y+w\not\in\Rad V^-$. Then
\begin{align*}
	[e,e^{-1}+y]\mu_v &= [e,e^{-1}+y+w]\alpha_v\zeta_w = [-(y+w)^{-1},0]\alpha_v\zeta_w = [v-(y+w)^{-1},0]\zeta_w \\
			&= [e,e^{-1}+((y+w)^{-1}-v)^{-1}]\zeta_w = [e,e^{-1}+((y+w)^{-1}-v)^{-1}+w]
\end{align*}
We need to show that $v-(y+w)^{-1}$ is invertible, and that $((y+w)^{-1}-v)^{-1}+w = -y^{-1}Q_v^{-1}$. For the first, assume $v-(y+w)^{-1}=x$ was not invertible. Then $y = (v-x)^{-1}-w$, and we find $y = 0\mod\Rad V$, so $y$ would not be invertible, a contradiction. For the identity, we get
\begin{align*}
	& ((y+w)^{-1}-v)^{-1}+w = -y^{-1}Q_w \\
\iff	& v-(y+w)^{-1} = (y^{-1}Q_w+w)^{-1} \\
\iff	& \bigl(v-(y+w)^{-1}\bigr)Q_{y^{-1}Q_w+w} = y^{-1}Q_w+w
\end{align*}
Now first observe that
\begin{align*}
	(y+w)Q_y^{-1} &= (y+w)^{-1}Q_{y+w}Q_y^{-1} \\
				  &= (y+w)^{-1}(Q_{y,w}+Q_y+Q_w)Q_y^{-1} \\
				  &= (y+w)^{-1}Q_{y,w}Q_y^{-1}+(y+w)^{-1} + (y+w)^{-1}Q_wQ_y^{-1}
\end{align*}
Now we get
\begin{align*}
	&\bigl(v-(y+w)^{-1}\bigr)Q_{y^{-1}Q_w+w} = \bigl(v-(y+w)^{-1}\bigr)(Q_{y^{-1}Q_w,w}+Q_{y^{-1}Q_w}+Q_w) \\
	&\;= \bigl(v-(y+w)^{-1}\bigr)(Q_{y^{-1}Q_w,w}+Q_wQ_y^{-1}Q_w+Q_w) \\
	&\;= vQ_{y^{-1}Q_w,w}+vQ_wQ_y^{-1}Q_w+vQ_w-(y+w)^{-1}(Q_{y^{-1}Q_w,w}+Q_wQ_y^{-1}Q_w+Q_w) \\
	&\;= vQ_{y^{-1}Q_w,w}+wQ_y^{-1}Q_w+w-(y+w)^{-1}Q_{y^{-1}Q_w,w}-(y+w)^{-1}Q_wQ_y^{-1}Q_w-(y+w)^{-1}Q_w \\
	&\;= vQ_wD_{w,y^{-1}}+wQ_y^{-1}Q_w+w-(y+w)^{-1}D_{y^{-1},w}Q_w-(y+w)^{-1}Q_wQ_y^{-1}Q_w-(y+w)^{-1}Q_w \\
	&\;= w+wD_{w,y^{-1}}+wQ_y^{-1}Q_w-(y+w)^{-1}D_{y^{-1},w}Q_w-(y+w)^{-1}Q_wQ_y^{-1}Q_w-(y+w)^{-1}Q_w \\
	&\;= w+2y^{-1}Q_w+wQ_y^{-1}Q_w-(y+w)^{-1}D_{y^{-1},w}Q_w-(y+w)^{-1}Q_wQ_y^{-1}Q_w-(y+w)^{-1}Q_w \\
	&\;= w+y^{-1}Q_w+\bigl(y^{-1}+wQ_y^{-1}-(y+w)^{-1}D_{y^{-1},w}-(y+w)^{-1}Q_wQ_y^{-1}-(y+w)^{-1}\bigr)Q_w \\
	&\;= w+y^{-1}Q_w+\bigl((y+w)Q_y^{-1}-(y+w)^{-1}D_{y^{-1},w}-(y+w)^{-1}Q_wQ_y^{-1}-(y+w)^{-1}\bigr)Q_w \\
	&\;= w+y^{-1}Q_w+\bigl((y+w)^{-1}Q_{y,w}Q_y^{-1}-(y+w)^{-1}D_{y^{-1},w}\bigr)Q_w \\
	&\;= w+y^{-1}Q_w
\end{align*}
This finishes the second case.

In the third case, we start with $[x,0]$ for $x\in \Rad V^+$. We get
\begin{align*}
	[x,0]\mu_{v} &= [x,0]\zeta_w\alpha_v\zeta_w = [x^w,0]\alpha_v\zeta_w = [x^w+v,0]\zeta_w \\
		&= [e,e^{-1}-(x^w+v)^{-1}]\zeta_w = [e,e^{-1}-(x^w+v)^{-1}+w]\;.
\end{align*}
We need to check that $x^w+v$ is invertible, and that $-(x^w+v)^{-1}+w = xQ_v^{-1}$. Since $x\in \Rad V^+$, we also have $x^w\in\Rad V^+$, so as $v$ is invertible, so is $x^w+v$. For the second, we need to show
\begin{align*}
		& -(x^w+v)^{-1}+w = xQ_v^{-1} \\
	\iff	& (x^w+v)^{-1} = w-xQ_v^{-1} \\
	\iff	& x^w = (w-xQ_v^{-1})^{-1}-v \\
	\iff	& x = \bigl((w-xQ_v^{-1})^{-1}-v\bigr)^{-w} \\
	\iff	& yQ_v = \bigl((w-y)^{-1}-v\bigr)^{-w} & \text{(set $x = yQ_v$)} \\
	\iff	& yQ_{v'} = \bigl(v'-(w'+y)^{-1}\bigr)^{w'} & \text{(set $w = -w'$ and $v = -v'$)}
\end{align*}
This is precisely the identity we have proven in the first case.
\end{proof}

Now we would like to use the permutations we have to construct a local Moufang set with Construction~\ref{constr:MUtau}. Of course, that requires the conditions for the construction to be satisfied:

\begin{proposition}
	Let $V$ be a local Jordan pair with invertible element $e\in V^+$. The group $U = \{\alpha_v\mid v\in V^+\}$ and permutation $\tau = \mu_e$ satisfy conditions \ref{itm:C1}, \ref{itm:C1'} and \ref{itm:C2} and we can take $0=[0,0]$, $\infty=[e,e^{-1}]$.
\end{proposition}
\begin{proof}
	The group $U$ fixes $[e,e^{-1}]$, as $0^v = (0-vQ_0)B_{0,v}^{-1} = (0-0)\id = 0$, hence we choose $\infty:=[e,e^{-1}]$. Furthermore, $\P(V)\setminus \overline{[e,e^{-1}]} = \{[x,0]\mid x\in V^+\}$, and $\alpha_v$ acts on this set by $x\mapsto x+v$. This action of $U$ on $\P(V)\setminus \overline{[e,e^{-1}]}$ is the regular representation of $(V^+,+)$, and hence a regular action. This proves \ref{itm:C1}.

	For $x\in V^\sigma$, denote $\overline{x}$ for the image of $x$ in the quotient $V^\sigma/\Rad V^\sigma$. Now $\overline{\P(V)}\setminus \{\overline{[e,e^{-1}]}\}$ has a natural correspondence to $\{[\overline{x},\overline{0}]\mid \overline{x}\in V^+/\Rad V^+\}$. The induced action of $\alpha_v$ on this set is given by $\overline{x}\mapsto\overline{x+v}$, which only depends on $\overline{x}$. The action of $\overline{U}$ on $\{[\overline{x},\overline{0}]\mid \overline{x}\in V^+/\Rad V^+\}$ is the regular representation of $(V^+/\Rad V^+,+)$, and hence a regular action. This shows \ref{itm:C1'}.

	By Proposition~\ref{prop:JordanMuaction}, we have $[e,e^{-1}]\tau = [0Q_e,0] = [0,0]$, which is not radically equivalent to $[e,e^{-1}]$. This means we can take $0:=[0,0]$. By the same proposition, $[0,0]\tau = [e,e^{-1}+0Q_e^{-1}] = [e,e^{-1}]$, which proves \ref{itm:C2}.
\end{proof}

\begin{definition}
	Let $V$ be a local Jordan pair with invertible element $e$. Using $U = \{\alpha_v\mid v\in V^+\}$ and $\tau = \mu_e$, we define $\M(V) := \M(U,\tau)$.
\end{definition}

\subsection{Proving that \texorpdfstring{$\M(V)$}{M(V)} is a local Moufang set}

As we have used Construction~\ref{constr:MUtau} to create $\M(V)$, we would like to use one of the equivalent conditions of Theorem~\ref{thm:constrMouf} to prove we have a local Moufang set. In order to do this, we need to know how the maps of type $\alpha_x$, $\gamma_x$ and $\mu_x$ correspond to the maps we have already defined.
Our notation in Definition~\ref{def:alphazetaJP} and Proposition~\ref{prop:JordanMuaction} suggests what this correspondence will be, and we make this precise in Proposition~\ref{prop:gammaJP} below.

In this subsection, we assume that we have a local Jordan pair with invertible element $e\in V^+$, and we set $U = \{\alpha_v\mid v\in V^+\}$ and $\tau = \mu_e$.

\begin{proposition}\label{prop:gammaJP}
	For all $v,t\in V^+$ with $t$ invertible, we have $\alpha_v^{\mu_t} = \zeta_{vQ_t^{-1}}$. Using this, we get $\alpha_{[v,0]} = \alpha_v$, $\gamma_{[v,0]} = \zeta_{vQ_e^{-1}}$ and $\mu_{[t,0]} = \mu_t$. Moreover, $-([t,0])\tau = -([t,0]\tau)$.
\end{proposition}
\begin{proof}
	We compute the action of $\alpha_v^{\mu_t}$ on the points of $\P(V)$ using Proposition~\ref{prop:JordanMuaction}. First, take $[e,e^{-1}+y]$ with $y\in V^-$. We get
	\begin{align*}
		[e,e^{-1}+y]\mu_t^{-1}\alpha_v\mu_t &= [yQ_t,0]\alpha_v\mu_t = [yQ_t+v,0]\mu_t = [e,e^{-1}+(yQ_t+v)Q_t^{-1}] \\
			&= [e,e^{-1}+y+vQ_t^{-1}] = [e,e^{-1}+y]\zeta_{vQ_t^{-1}}\;.
	\end{align*}
	Next, take $x\in\Rad V^+$, then $xQ_t^{-1}\in\Rad V^-$, so
	\begin{align*}
		[x,0]\mu_t^{-1}\alpha_v\mu_t &= [e,e^{-1}+xQ_t^{-1}]\alpha_v\mu_t = [e,e^{-1}+(xQ_t^{-1})^v]\mu_t = [e,e^{-1}+(xQ_t^{-1})^v]\mu_t \\
			&= [(xQ_t^{-1})^vQ_t,0] = [x^{vQ_t^{-1}},0] = [x,0]\zeta_{vQ_t^{-1}}\;,
	\end{align*}
	where we used Proposition~\ref{prop:JPbasic}\ref{itm:JPbasicQy}. Hence for all points of $\P(V)$, the image of $\zeta_{vQ_y}$ is equal to that of $\alpha_v^{\mu_t}$, so these permutations are equal.

	For the other statements, observe first that $\alpha_{[v,0]} = \alpha_v$ since $\alpha_v$ is the unique element of $U$ mapping $[0,0]$ to $[v,0]$, and by definition $\gamma_{[v,0]} = \alpha_{[v,0]}^\tau = \alpha_v^{\mu_e} = \zeta_{vQ_e^{-1}}$. Finally, if $t$ is invertible, we have
	\[(-[t,0])\tau = [-t,0]\tau = [e,e^{-1}+t^{-1}]\tau = [t^{-1}Q_e,0]\]
	and similarly $-([t,0]\tau) = [t^{-1}Q_e,0]$, which shows the last statement. Using the definition of $\mu_{[t,0]}$ in Construction~\ref{constr:MUtau}, we get
	\[\mu_{[t,0]} := \gamma_{(-[t,0])\tau^{-1}}\alpha_{[t,0]}\gamma_{-([t,0]\tau^{-1})} = \gamma_{[t^{-1}Q_e,0]}\alpha_{[t,0]}\gamma_{[t^{-1}Q_e,0]} = \zeta_{t^{-1}}\alpha_t\zeta_{t^{-1}} = \mu_t\;.\qedhere\]
\end{proof}

As we now know what all the maps of Construction~\ref{constr:MUtau} are, we can use them to show $\M(V)$ is a local Moufang set.

\begin{theorem}\label{thm:MV}
	Let $V$ be a local Jordan pair with invertible element $e$. Set $U:=\{\alpha_v\mid v\in V^+\}$ and $\tau = \mu_e$,
    where $\alpha_v$ and $\mu_e$ are as in Definition~\ref{def:alphazetaJP} and Proposition~\ref{prop:JordanMuaction}, respectively.
    Then $\M(V) = \M(U,\tau)$ is a local Moufang set.
\end{theorem}
\begin{proof}
	In Construction~\ref{constr:MUtau}, we have $U_0 := U^\tau = \{\zeta_{vQ_e^{-1}}\mid v\in V^+\} = \{\zeta_w \mid w\in V^-\}$, where the final equality follows from the fact that $Q_e$ is invertible. Now let $[t,0]$ be an arbitrary unit in $\P(V)$, then $\mu_{[t,0]} = \mu_t$, so
	\[U^{\mu_{[t,0]}} = \{\alpha_v^{\mu_t}\mid v\in V^+\} = \{\zeta_{vQ_t^{-1}}\mid v\in V^+\} = \{\zeta_w\mid w\in V^-\} = U_0\;,\]
	since $t$, and hence $Q_t$ is invertible. Hence $U_0 = U_\infty^{\mu_{[t,0]}}$ for all units $[t,0]$, and Construction~\ref{constr:MUtau} gives a local Moufang set by Theorem~\ref{thm:constrMouf}.
\end{proof}

\section{From local Moufang sets to local Jordan pairs}\label{sec:LMStoJP}

\subsection{The construction and basic properties}

We now investigate the reverse construction: we try to make a local Jordan pair starting from a local Moufang set satisfying some additional assumptions.
One obvious necessary assumption is that the root groups have to be abelian, and by Proposition~\ref{prop:gammaJP}, we also know that the local Moufang set has to be special.
We will also impose a restriction to avoid the cases where $V/\Rad V$ has characteristic $2$ or $3$.
Finally, we will need a linearity assumption.

Notice that for a given Jordan pair $V$, the Moufang set $\M(V)$ cannot detect the base ring~$k$ over which the Jordan pair was initially defined.
For this reason, the Jordan pair that we will (try to) construct will be defined over the base ring $\Z$, i.e., it will consist of a pair of $\Z$-modules.

\begin{construction}\label{constr:Jpair}
	Suppose $\M$ is a local Moufang set satisfying the following properties:
	\begin{enumerate}[label=\textnormal{(J\arabic*)}]
		\item $\M$ is special;\label{itm:J1}
		\item $U_\infty$ is abelian;\label{itm:J2}
		\item if $x$ is a unit, then so is $x\cdot2$ and $x\cdot 3$.\label{itm:J3}
	\end{enumerate}
	Then we define two $\Z$-modules as follows:
	\begin{itemize}
		\item $V^+ := X\setminus\overline{\infty}$ with $x+z:=0\alpha_x\alpha_z$;
		\item $V^- := X\setminus\overline{0}$ with $y\plustil w:=\infty\gamma_{y\tau}\gamma_{w\tau}$.
	\end{itemize}
	Now we have, for all $x,z\in V^+$ and units $t$, $(x+z)\mu_t = x\mu_t\plustil z\mu_t$, and similarly for all $y,w\in V^-$, $(y\plustil w)\mu_t = y\mu_t+ w\mu_t$. Hence $\mu$-maps are group isomorphisms between $V^+$ and $V^-$. We now define the following maps:
	\begin{alignat}{2}
		\mu_{x,z} &:= \mu_{x+z}-\mu_x-\mu_z:V^-\to V^+&&\text{for units $x,z\in V^+$ such that $x+z$ is a unit;}\label{map:muxz} \\
		\mutil_{y,w} &:= \mu_{y\plustil w}\mintil\mu_y\mintil\mu_w:V^+\to V^- \quad &&\text{for units $y,w\in V^+$ such that $y\plustil w$ is a unit.}\label{map:muyw}
	\end{alignat}
	The final assumption we make is the following:
	\begin{enumerate}[label=\textnormal{(J\arabic*)},resume]
		\item There are bilinear maps
		\[\mu_{\cdot,\cdot}\colon V^+ \times V^+ \to \Hom(V^-, V^+) \text{ and }\mutil_{\cdot,\cdot}\colon V^- \times V^- \to \Hom(V^+, V^-)\]
		that extend \eqref{map:muxz} and \eqref{map:muyw}. \label{itm:J4}
	\end{enumerate}
	We now have a pair of $\Z$-modules $(V^+,V^-)$ and bilinear maps $\mu_{\cdot,\cdot}$ and $\mutil_{\cdot,\cdot}$ which will define a local Jordan pair, as will be shown in Theorem~\ref{thm:localJP}.
\end{construction}
\begin{remark}
	By \hyperref[itm:J1]{(J1-2)}, $\tau$ is an involution, so $\gamma_{y\tau}$ is the unique element of $U_0$ mapping $\infty$ to $y$, and hence it does not depend on $\tau$. In particular, $y\plustil w$ does not depend on the choice of $\tau$.
\end{remark}

\begin{remark}\label{rem:expression_muxz}
	If we have a local Moufang set satisfying \hyperref[itm:J1]{(J1-4)}, we can express $\mu_{x,z}$ (and similarly $\mutil_{y,w}$) in terms of $\mu$-maps for any pair of $x,z\in V^+$ by the linearity:
	\begin{align*}
		\mu_{x,z} &= \mu_{x+z}-\mu_x-\mu_z &&\text{if $x$, $z$ and $x+z$ are units;} \\
		\mu_{x,z} &:= -\mu_{-x,z} &&\text{if $x$, $z$ are units but $x+z$ is not a unit;} \\
		\mu_{x,z} &:= \mu_{x,x+z}-\mu_{x,x} &&\text{if $x$ is a unit but $z$ is not a unit;} \\
		\mu_{x,z} &:= \mu_{z,x} &&\text{if $z$ is a unit but $x$ is not a unit;} \\
		\mu_{x,z} &:= \mu_{x+e,z}-\mu_{e,z} &&\text{if $x$, $z$ are not units, and $e$ is an arbitrary unit.}
	\end{align*}
\end{remark}

For the remainder of this section, we will always assume that we have a local Moufang set satisfying \hyperref[itm:J1]{(J1-4)}.
We start by showing some basic identities, which will help to show that the construction gives us a Jordan pair.

\begin{lemma}\label{lem:JMS_identities}
	In a local Moufang set where \hyperref[itm:J1]{\normalfont{(J1-4)}} holds, we have the following identities:
	\begin{enumerate}
		\item $t\mu_{t,x} = -x\cdot 2$ \quad for all units $t$ and all $x\in V^+$;\label{itm:xxy}
		\item $y\mu_{t,t} = y\mu_t\cdot 2$ \quad for all units $t$ and all $y\in V^-$\label{itm:xyx};
		\item $\mu_s\mu_{s,t}\mu_t = \mu_t\mu_{s,t}\mu_s = \mutil_{s,t}$ \quad for all units $s,t$;\label{itm:reverse_order}
		\item $s\mu_t\mu_{s,t} = -t\mu_s\cdot 2$ \quad for all units $s,t$;\label{itm:xxmuyy}
		\item $y\mu_{x,z}\tau = y\tau\mutil_{x\tau,y\tau}$ \quad for all $x,z\in V^+$ and $y\in V^-$.
	\end{enumerate}
\end{lemma}
\begin{proof}\leavevmode
	\begin{enumerate}
		\item First assume $x,t$ are units such that $x+t$ is also a unit. Then by Lemma~\ref{lem:specialsum}, we have
		\[x\mu_{x+t} = -t\cdot2-x+x\mu_t\;,\]
		hence
		\[x\mu_{x,t} = -t\cdot2-x+x\mu_t - x\mu_x - x\mu_t = -t\cdot2\;,\]
		since $x\mu_x = -x$. If $x+t$ is not a unit, we can replace $t$ by $-t$ (as $t\cdot2$ is a unit by \ref{itm:J3}, $x-t=x+t-(t\cdot 2)$ is a unit) and get
		\[x\mu_{x,t} = -x\mu_{x,-t} = -(t\cdot 2) = -t\cdot 2\;.\]
		Finally, if $t$ is not a unit, we can take any unit $e$ and get
		\[x\mu_{x,t} = x\mu_{x,t-e}+x\mu_{x,e} = -(t-e)\cdot 2-e\cdot 2 = -t\cdot2\;.\]
		\item We have $\mu_{t,t} = \mu_{t\cdot 2}-\mu_t\cdot 2$. By Proposition~\ref{prop:mu-x.s}\ref{itm:mu-x.s}, we get $y\mu_{t\cdot 2} = y\mu_t\cdot 4$, so
		\[y\mu_{t,t} = y\mu_t\cdot 4 - y\mu_t\cdot 2 = y\mu_t\cdot 2\;.\]
		\item Assume first that $s+t$ is also a unit. By Proposition~\ref{prop:mu-involution}\ref{itm:mucommute}, we then have $\mu_s\mu_{s+t}\mu_t = \mu_t\mu_{s+t}\mu_s = \mu_{(s\tau+t\tau)\tau}$. By the linearity of $\tau$, we have $\mu_{(s\tau+t\tau)\tau} = \mu_{s\plustil t}$, so we can now use the definitions of $\mu_{s,t}$ and $\mutil_{s,t}$ to get
		\begin{align*}
			\mu_s(\mu_{s,t}+\mu_s+\mu_t)\mu_t &= \mu_t(\mu_{s,t}+\mu_s+\mu_t)\mu_s = \mutil_{s,t}\plustil\mu_s\plustil\mu_t\;,
			\intertext{so since all $\mu$-maps are involutions,}
			\mu_s\mu_{s,t}\mu_t \plustil \mu_t \plustil \mu_s &= \mu_t\mu_{s,t}\mu_s \plustil \mu_t \plustil \mu_s = \mutil_{s,t}\plustil\mu_s\plustil\mu_t\;,
			\intertext{hence}
			\mu_s\mu_{s,t}\mu_t &= \mu_t\mu_{s,t}\mu_s = \mutil_{s,t}\;.
		\end{align*}
		If $s+t$ is not a unit, we can replace $t$ by $-t$ and use linearity the get the result.
		\item From \ref{itm:xxy} we know $-s\mu_{s,t} = t\cdot 2$. Hence, using \ref{itm:reverse_order}, we get
		\[t\cdot 2 = s\mu_s\mu_{s,t} = s\mu_t\mu_{s,t}\mu_s\mu_t\;.\]
		Applying $\mu_t\mu_s$ to both sides, we get
		\[t\mu_t\mu_s\cdot 2 = s\mu_t\mu_{s,t}\quad\text{ and hence }\quad s\mu_t\mu_{s,t} = -t\mu_s\cdot 2\;.\]
		\item Assume first that $x$, $z$ and $x+z$ are units. Then we have
		\begin{align*}
			y\mu_{x,z}\tau &= y(\mu_{x+z}-\mu_x-\mu_z)\tau \\
				&= y\tau(\mu_{(x+z)\tau}\mintil\mu_{x\tau}\mintil\mu_{z\tau}) \\
				&= y\tau(\mu_{x\tau\plustil z\tau}\mintil\mu_{x\tau}\mintil\mu_{z\tau}) \\
				&= y\tau\mutil_{x\tau,z\tau}\;.
		\end{align*}
		By linearity, this identity now holds for all $x$ and $z$ in $V^+$.\qedhere
	\end{enumerate}
\end{proof}

\begin{remark}
	Since $V^+$ and $V^-$ play the same role in the construction (our choice of $0$ and $\infty$ could have been reversed), any identity we have proven will also hold with $+$ and $-$ interchanged. For example, the identity $t\mutil_{t,y} = \mintil y\cdottil 2$ also holds for all units $t$ and all $y\in V^-$.
\end{remark}

In the next two subsections, we will prove the axioms \ref{axiom:JP1} and \ref{axiom:JP2} of a Jordan pair. In the process, we will also show the linearizations of those axioms, so by Proposition~\ref{prop:JPsufficientaxioms}\ref{itm:sufficient_JP}, we will then have shown that we have a Jordan pair, since assumption \ref{itm:J3} says in particular that there is no $2$-torsion.
To prove \ref{axiom:JP1} and \ref{axiom:JP2}, we will first restrict everything to units.

\subsection{Proving the axioms for units}

We first prove \ref{axiom:JP1} by linearizing some of the basic identities we have shown earlier. The proof is based on ideas from the proof of Theorem 5.11 in \cite{MR2425693}. Remark that $\mu$-maps will correspond to the quadratic maps of the Jordan pair, and $\mu_{\cdot,\cdot}$ (and $\mutil_{\cdot,\cdot}$) will correspond to the bilinearizations $Q_{\cdot,\cdot}$. In these terms, \ref{axiom:JP1} translates to $yQ_{x,zQ_x} = xQ_{y,z}Q_x$, or in the local Moufang set: $y\mutil_{x,z\mu_x} = x\mu_{y,z}\mu_x$. Up to renaming, this is the identity we will prove for units:

\begin{proposition}\label{prop:JMS_JP1_units}
	In a local Moufang set where \hyperref[itm:J1]{\normalfont{(J1-4)}} holds, we have the following identities:
	\begin{enumerate}
		\item $\mu_{r,s}\mu_s\mu_{r,t}+\mu_{t,s}\mu_s\mu_r = \mu_{r,t}\mu_s\mu_{r,s}+\mu_r\mu_s\mu_{t,s}$ \quad for all units $r,s,t$;\label{itm:linearize_commuting}
		\item $r\mu_s\mu_{t,s}+t\mu_s\mu_{r,s} = -s\mu_{r,t}\cdot 2$ \quad for all units $r,s,t$;\label{itm:linearize_xxmuyy}
		\item $r\mu_s\mu_{t,s} = t\mutil_{r,s}\mu_r$ \quad for all units $r,s,t$;\label{itm:JMS_JP1_units_iii}
		\item $x\mutil_{z\mu_y,y} = y\mu_{x,z}\mu_y = z\mutil_{x\mu_y,y}$ \quad for all units $x,z\in V^+$ and all units $y\in V^-$.\label{itm:JP1_units}
	\end{enumerate}
\end{proposition}
\begin{proof}\leavevmode
	\begin{enumerate}
		\item We start with Lemma~\ref{lem:JMS_identities}\ref{itm:reverse_order}, conjugating both sides by $\mu_r$ and then replacing $r$ by $r+t\cdot \ell$, for those $\ell\in\{1,2,3\}$ for which $r+t\cdot \ell$ is a unit. Since $\mu_{r+t\cdot \ell} = \mu_{r,t}\cdot \ell+\mu_r+\mu_t\cdot \ell^2$, we get
		\[(\mu_{r,s}+\mu_{t,s}\cdot \ell)\mu_s(\mu_{r,t}\cdot \ell+\mu_r+\mu_t\cdot \ell^2) = (\mu_{r,t}\cdot \ell+\mu_r+\mu_t\cdot \ell^2)\mu_s(\mu_{r,s}+\mu_{t,s}\cdot \ell)\]
		and after expanding,
		\begin{align*}
			\mu_{r,s}\mu_s\mu_r + (\mu_{r,s}\mu_s\mu_{r,t}+\mu_{t,s}\mu_s\mu_r)\cdot \ell + (\mu_{t,s}\mu_s\mu_{r,t}+\mu_{r,s}\mu_s\mu_t)\cdot \ell^2 + \mu_{t,s}\mu_s\mu_t\cdot \ell^3 \\
			= \mu_r\mu_s\mu_{r,s} + (\mu_{r,t}\mu_s\mu_{r,s}+\mu_r\mu_s\mu_{t,s})\cdot \ell + (\mu_{r,t}\mu_s\mu_{t,s}+\mu_t\mu_s\mu_{r,s})\cdot \ell^2 + \mu_t\mu_s\mu_{t,s}\cdot \ell^3
		\end{align*}
		Observe that the constant terms and terms with $\ell^3$ cancel due to Lemma~\ref{lem:JMS_identities}\ref{itm:reverse_order}, so we have
		\begin{align*}
			&(\mu_{r,s}\mu_s\mu_{r,t}+\mu_{t,s}\mu_s\mu_r)\cdot \ell + (\mu_{t,s}\mu_s\mu_{r,t}+\mu_{r,s}\mu_s\mu_t)\cdot \ell^2 \\
			= {}& (\mu_{r,t}\mu_s\mu_{r,s}+\mu_r\mu_s\mu_{t,s})\cdot \ell + (\mu_{r,t}\mu_s\mu_{t,s}+\mu_t\mu_s\mu_{r,s})\cdot \ell^2\;.
		\end{align*}
		Observe now that there are at least two values of $\ell$ for which $r+t\cdot \ell$ is a unit, since $t$ and $t\cdot2$ are units and adding a unit to a non-unit gives a unit. Using those two values, we can deduce that the coefficients on the left and right hand side of both $\ell$ and $\ell^2$ are equal. This means that
		\[\mu_{r,s}\mu_s\mu_{r,t}+\mu_{t,s}\mu_s\mu_r = \mu_{r,t}\mu_s\mu_{r,s}+\mu_r\mu_s\mu_{t,s}\;.\]
		\item We similarly linearize Lemma~\ref{lem:JMS_identities}\ref{itm:xxmuyy}, replacing $r$ by $r+t$ if $r+t$ is a unit (if not, replace $r$ by $r-t$). We get
		\begin{align*}
			(r+t)\mu_s\mu_{r+t,s} &= -s\mu_{r+t}\cdot 2 \;, \\
		\intertext{so using \ref{itm:J4} and the definition of $\mu_{r,t}$,}
			r\mu_s\mu_{r,s}+r\mu_s\mu_{t,s}+t\mu_s\mu_{r,s}+t\mu_s\mu_{t,s} &= -s\mu_{r,t}\cdot 2 - s\mu_r\cdot 2 - s\mu_t\cdot 2\;. \\
		\intertext{We can now use Lemma~\ref{lem:JMS_identities}\ref{itm:xxmuyy} twice to get}
			r\mu_s\mu_{t,s}+t\mu_s\mu_{r,s} &= -s\mu_{r,t}\cdot 2\;.
		\end{align*}
		\item We apply identity \ref{itm:linearize_commuting} to the element $r$, and use Lemma~\ref{lem:JMS_identities} to get
		\begin{align*}
			(-s\cdot 2)\mu_s\mu_{r,t} + r\mu_{t,s}\mu_s\mu_r &= (-t\cdot 2)\mu_s\mu_{r,s} + (-r)\mu_s\mu_{t,s} \;. \\
		\intertext{Using linearity, this yields}
			s\mu_{r,t}\cdot 2 + r\mu_{t,s}\mu_s\mu_r &= -t\mu_s\mu_{r,s}\cdot 2 -r\mu_s\mu_{t,s} \;, \\
		\intertext{and by \ref{itm:linearize_xxmuyy},}
			r\mu_{t,s}\mu_s\mu_r &= -t\mu_s\mu_{r,s} \,. \\
		\intertext{We now replace $r$ by $r\mu_s$ and apply $\mu_s\mu_r$:}
			r\mu_s\mu_{t,s}\mu_r\mu_s\mu_s\mu_r &= -t\mu_s\mu_{r\mu_s,s}\mu_s\mu_r \\
			\implies\quad r\mu_s\mu_{t,s} &= -t\mutil_{r,s\mu_s}\mu_r \\
			\implies\quad r\mu_s\mu_{t,s} &= t\mutil_{r,s}\mu_r\;.
		\end{align*}
		\item By Lemma~\ref{lem:JMS_identities}\ref{itm:reverse_order}, we have $\mutil_{r,s}\mu_r = \mu_s\mu_{r,s}$, so \ref{itm:JMS_JP1_units_iii} becomes
		\[ r\mu_s\mu_{t,s} = t\mu_s\mu_{r,s}\;.\]
		We can now plug this in \ref{itm:linearize_xxmuyy} and use the unique $2$-divisibility to get
		\[r\mu_s\mu_{t,s} = -s\mu_{r,t}\;.\]
		We now apply $\mu_s$ to both sides and use linearity to get
		\[s\mu_{r,t}\mu_s = \mathord{\mintil}r\mu_s\mu_{t,s}\mu_s = \mathord{\mintil}r\mutil_{t\mu_s,\mathord{\mintil}s} = r\mutil_{t\mu_s,s}\;.\]
		After renaming variables, we get the first identity we wanted to prove. For the second identity, remark that $y\mu_{x,z}\mu_y$ is symmetric in $x$ and $z$, hence
		\[ x\mutil_{z\mu_y,y} = y\mu_{x,z}\mu_y = z\mutil_{x\mu_y,y}\;.\qedhere\]
	\end{enumerate}
\end{proof}

\begin{remark}
	The technique used in the previous lemma will be used extensively to linearize many identities which hold when all unknowns are units. We describe it here in generality: replace an unknown $x$ by $x+\hat{x}\cdot \ell$ for some unused variable name $\hat{x}$, and $\ell\in\{1,2,3,4\}$. Next, we can combine several facts to expand the resulting identity as a polynomial in powers of $\ell$:
	\begin{itemize}
		\item the linearity of $\mu_{\cdot,\cdot}$ and $\mutil_{\cdot,\cdot}$;
		\item the definition of $\mu_{\cdot,\cdot}$ to expand $\mu_{x+\hat{x}\cdot\ell} = \mu_{x,\hat{x}}\cdot\ell + \mu_x + \mu_{\hat{x}\cdot\ell}$, which requires $x+\hat{x}\cdot\ell$ to be a unit;
		\item the identity $\mu_{\hat{x}\cdot\ell} = \mu_{\hat{x}}\cdot\ell^2$, which requires $\hat{x}$ to be a unit.
	\end{itemize}
	We assume the highest power of $\ell$ occurring is $\ell^4$. By the identity we started with, the coefficients of $\ell^0$ and $\ell^4$ will always be equal. If we now find $3$ values for which $x+\hat{x}\cdot \ell$ is a unit, we can solve the Vandermonde system of equations and then we know that the coefficients of $\ell^1$, $\ell^2$ and $\ell^3$ are also equal.

	This technique will be used in many proofs to come, but it does not necessarily work for any identity in $\mu_\cdot$, $\mu_{\cdot,\cdot}$ and $\mutil_{\cdot,\cdot}$. It can be checked that it does work whenever it is used.
\end{remark}

Next, we will prove \ref{axiom:JP2} for units. This axiom for Jordan pairs corresponds to the Triple Shift Formula for Jordan algebras (\cite[p.~202]{TasteOfJordanAlgebras}), which can be deduced from the axioms of Jordan algebras. We will use ideas from \cite{MR0325715} where such a deduction is made, and adapt them to the context of local Moufang sets. This will require many intermediate identities and will also require the choice of a fixed invertible element of the Jordan-pair-to-be. Hence we fix a unit $e$ of our local Moufang set, which we will use throughout the following few lemmas. We begin with two basic consequences of Proposition~\ref{prop:JMS_JP1_units}.
\begin{lemma}\label{lem:JMS_basic_identities}
	In a local Moufang set where \hyperref[itm:J1]{\normalfont{(J1-4)}} holds, we have the following identities:
	\begin{enumerate}
		\item $y\mu_{x,e} = e\mutil_{y,x\mu_e}\mu_e = -e\mu_{y\mu_e,x}$ \quad for all units $x\in V^+$ and $y\in V^-$; \label{itm:muxe}
		\item $x\mu_e\mu_{x,e} = -e\mu_x\cdot 2$ \quad for all units $x\in V^+$. \label{itm:xmuemuxe}
	\end{enumerate}
\end{lemma}
\begin{proof}\leavevmode
	\begin{enumerate}
		\item We take Proposition~\ref{prop:JMS_JP1_units}\ref{itm:JP1_units}, interchange the roles of $V^+$ and $V^-$ and set $s=e$ to get $r\mu_{t\mu_e,e} = e\mutil_{r,t}\mu_e$. Now replace $r$ by $y$ and $t$ by $x\mu_e$ to get the first equality. For the second, we use
		\[e\mutil_{y,x\mu_e}\mu_e = e\mu_e\mu_{y\mu_e,x\mu_e\mu_e} = (\mathord{\mintil}e)\mu_{y\mu_e,x} = -e\mu_{y\mu_e,x}\;.\]
		\item Setting $y = x\mu_e$ in \ref{itm:muxe}, we get $x\mu_e\mu_{x,e} = -e\mu_{x,x} = -e\mu_x\cdot 2$.\qedhere
	\end{enumerate}
\end{proof}

We start building up some identities that we will use to prove \ref{axiom:JP2} for units.

\begin{lemma}\label{lem:JMS_QJ-identities}
	Let $\M$ be a local Moufang set satisfying \hyperref[itm:J1]{\normalfont{(J1-4)}}.
    Then for all units $x,z,v \in V^+$ and all units $y,w \in V^-$, the following identities hold:
	\begin{enumerate}
		\item $\mu_x\mu_y\mu_z+\mu_z\mu_y\mu_x+\mu_{x,z}\mu_y\mu_{x,z} = \mu_{y\mu_{x,z}} + \mu_{y\mu_x,y\mu_z}$; \label{QJ7}
		\item $\mu_{x,e}\mu_e\mu_{x,e}+\mu_{e\mu_x,e} = \mu_x\cdot 2$; \label{QJ20}
		\item $y\mu_{x,w\mu_{x,z}}+y\mu_{z,w\mu_x} = x\mutil_{y,w}\mu_{x,z} + z\mutil_{y,w}\mu_x$; \label{QJ9}
		\item $e\mu_{z,e\mu_x} = x\mu_e\mu_{z,x}$; \label{QJ26}
		\item $e\mu_{v,e\mu_{x,z}} = x\mu_e\mu_{v,z}+z\mu_e\mu_{v,x}$. \label{QJ27}
	\end{enumerate}
\end{lemma}
\begin{proof}\leavevmode
	\begin{enumerate}
		\item We start from the identity $\mu_x\mu_y\mu_x = \mu_{y\mu_x}$, and linearize $y$ to $y\plustil y'\cdottil\ell$.
        Equating the coefficients of $\ell$ on both sides yields
		\[\mu_x\mu_{y,y'}\mu_x = \mu_{y\mu_x,y'\mu_x}\;.\]
		Next, we linearize $x$ to $x+z\cdot\ell$, and equate the coefficients of $\ell^2$ on both sides of the equality; this gives
		\[\mu_x\mu_{y,y'}\mu_z+\mu_z\mu_{y,y'}\mu_x+\mu_{x,z}\mu_{y,y'}\mu_{x,z} = \mu_{y\mu_x,y'\mu_z}+\mu_{y\mu_z,y'\mu_x}+\mu_{y\mu_{x,z},y'\mu_{x,z}}\;.\]
		We can now set $y'=y$ and use Lemma~\ref{lem:JMS_identities}\ref{itm:xyx} to get
		\[\mu_x\mu_y\mu_z\cdot 2+\mu_z\mu_y\mu_x\cdot 2+\mu_{x,z}\mu_y\mu_{x,z}\cdot 2 = \mu_{y\mu_x,y\mu_z}\cdot 2+\mu_{y\mu_{x,z}}\cdot 2\;.\]
		The unique $2$-divisibility now gives us the desired identity.
		\item We set $y=z=e$ in \ref{QJ7} and get
		\[\mu_x\mu_e\mu_e + \mu_e\mu_e\mu_x + \mu_{x,e}\mu_e\mu_{x,e} = \mu_{e\mu_{x,e}} + \mu_{e\mu_x,e\mu_e}\;,\]
		which reduces to
		\[\mu_x\cdot 2 + \mu_{x,e}\mu_e\mu_{x,e} = \mu_{-x\cdot 2} + \mu_{e\mu_x,-e}\]
		and by $\mu_{-x\cdot 2} = \mu_x\cdot 4$ and linearity, we get
		\[\mu_x\cdot 2 + \mu_{x,e}\mu_e\mu_{x,e} = \mu_{x}\cdot 4 - \mu_{e\mu_x,e}\;,\]
		so after rearranging we get the identity we wanted to prove.
		\item Starting from the first equality of Proposition~\ref{prop:JMS_JP1_units}\ref{itm:JP1_units}, we interchange the roles of $V^+$ and $V^-$ and rename some variables to get $y\mu_{w\mu_x,x} = x\mutil_{y,w}\mu_x$. Next, we linearize $x$ to $x+z\cdot\ell$; the coefficients of $\ell^1$ give the desired equality.
		\item Set $x=e$ and $w = z\mu_e$ in \ref{QJ9} to get
		\[y\mu_{e,z\mu_e\mu_{e,z}}+y\mu_{z,z} = e\mutil_{y,z\mu_e}\mu_{e,z} + z\mutil_{y,z\mu_e}\mu_e\;.\]
		By Lemma~\ref{lem:JMS_basic_identities}\ref{itm:xmuemuxe}, $z\mu_e\mu_{e,z} = -e\mu_z\cdot 2$, so the previous identity becomes
		\begin{equation}
			-y\mu_{e,e\mu_z}\cdot 2+y\mu_z\cdot 2 = e\mutil_{y,z\mu_e}\mu_{e,z} + z\mu_e\mu_{y\mu_e,z}\;. \label{eq:QJ26_stepi}
		\end{equation}
		Next, we take identity \ref{QJ20}, replace $x$ by $z$, and apply it to $y$. This gives
		\[y\mu_{z,e}\mu_e\mu_{z,e}+y\mu_{e\mu_z,e} = y\mu_z\cdot 2\;,\]
		which we can combine with \eqref{eq:QJ26_stepi} to
		\[-y\mu_{e,e\mu_z}+y\mu_{z,e}\mu_e\mu_{z,e} = e\mutil_{y,z\mu_e}\mu_{e,z} + z\mu_e\mu_{y\mu_e,z}\;.\]
		By Lemma~\ref{lem:JMS_basic_identities}\ref{itm:muxe}, we have
		\[y\mu_{z,e} = e\mutil_{y,z\mu_e}\mu_e\;,\]
		so
		\[y\mu_{z,e}\mu_e\mu_{z,e} = e\mutil_{y,z\mu_e}\mu_{e,z}\;.\]
		From this, we get
		\[-y\mu_{e,e\mu_z} = z\mu_e\mu_{y\mu_e,z}\;.\]
		Again by Lemma~\ref{lem:JMS_basic_identities}\ref{itm:muxe}, we have
		\[-y\mu_{e,e\mu_z} = -(-e\mu_{y\mu_e,e\mu_z}) = e\mu_{y\mu_e,e\mu_z}\;.\]
		Combining these last two identities, replacing $z$ by $x$ and $y$ by $z\mu_e$ gives the desired identity.
		\item We linearize $x$ to $x+v\cdot\ell$ in \ref{QJ26}, take the coefficients in $\ell^1$ and interchange $z$ and $v$ to get the desired identity.\qedhere
	\end{enumerate}
\end{proof}

We are now ready to prove \ref{axiom:JP2} for units. Our starting point is an identity which is symmetric in two unknowns on one side of the equality sign, and hence must also be symmetric in those unknowns on the other side.
For clarity in the notation, we will occasionally write $\mu(\cdot,\cdot)$ instead of $\mu_{\cdot,\cdot}$.

\begin{proposition}\label{prop:JMS_JP2_units}
	In a local Moufang set where \hyperref[itm:J1]{\normalfont{(J1-4)}} holds, we have the following identities:
	\begin{enumerate}[itemsep=.4ex]
		\item $e\mu(z\mu_e\mu_x,z) = e\mu(x\mu_e\mu_z,x)$ \quad for all units $x,z\in V^+$;\label{QJ29}
		\item $e\mu(z\mu_e\mu_{x,v},z) = e\mu(x\mu_e\mu_z,v)+e\mu(v\mu_e\mu_z,x)$ \quad for all units $x,z\in V^+$ and any $v\in V^+$;\label{QJ34}
		\item $e\mu(x\mu_e\mu_z,x\mu_e\mu_{v,e}) + e\mu(x\mu_e\mu_{v,e}\mu_e\mu_z,x) = e\mu_{z,v}\mu_e\mu_x\mu_e\mu_{z,e} + e\mu(z\mu_e\mu_x,v)\mu_e\mu_{z,e}$ \begin{flushright} for all units $x,z,v\in V^+$; \end{flushright} \label{QJ33}
		\item $v\mu_e\mu(z\mu_e\mu_x,z) - v\mu_e\mu(x\mu_e\mu_z,x) = e\mu(x\mu_e\mu_z,x\mu_e\mu_{v,e}) - e\mu(z\mu_e\mu_x,v)\mu_e\mu_{z,e}$ \begin{flushright} for all units $x,z,v\in V^+$; \end{flushright} \label{QJ32}
		\item $x\mu_y\mu_{x,z} = y\mu_{y\mu_x,z}$ \quad for all units $x,z\in V^+$ and all units $y\in V^-$.\label{itm:JP2_units}
	\end{enumerate}
\end{proposition}
\begin{proof}\leavevmode
	\begin{enumerate}
		\item We start with Proposition~\ref{prop:JMS_JP1_units}\ref{itm:JP1_units}, where we interchange the roles of $V^+$ and $V^-$, set $x=e$ and rename the other variables variables:
		\[e\mu_{y\mu_z,z} = z\mutil_{e,y}\mu_z\;.\]
		Linearizing $z$ to $z+x\cdot\ell$ and taking the coefficients of $\ell^1$ gives us
		\[e\mu_{y\mu_{x,z},x}+e\mu_{y\mu_x,z} = x\mutil_{e,y}\mu_{x,z}+z\mutil_{e,y}\mu_x\;.\]
		Substituting $z\mu_e$ for $y$, we get
		\[e\mu(z\mu_e\mu_{x,z},x)+e\mu(z\mu_e\mu_x,z) = x\mutil_{e,z\mu_e}\mu_{x,z}+z\mutil_{e,z\mu_e}\mu_x\;.\]
		We want to show that $e\mu(z\mu_e\mu_x,z)$ is symmetric in $x$ and $z$, i.e.\ we need to show that the remaining terms are symmetric in $x$ and $z$. By Proposition~\ref{prop:JMS_JP1_units}\ref{itm:JP1_units},
		\[x\mutil_{e,z\mu_e}\mu_{x,z} = e\mu_{x,z}\mu_e\mu_{x,z}\;,\]
		so this term is symmetric. Hence it remains to show that $z\mutil_{e,z\mu_e}\mu_x-e\mu(z\mu_e\mu_{x,z},x)$ is symmetric.
		We have
		\begin{align*}
		z\mutil_{e,z\mu_e}\mu_x-e\mu(z\mu_e\mu_{x,z},x)
			&= e\mu_{z,z}\mu_e\mu_x - e\mu(e\mu_{x,e\mu_z},x) = e\mu_z\mu_e\mu_x\cdot 2 - e\mu(e\mu_{x,e\mu_z},x) \\
		\intertext{by Proposition~\ref{prop:JMS_JP1_units}\ref{itm:JP1_units} and Lemma~\ref{lem:JMS_QJ-identities}\ref{QJ26}. By Lemma~\ref{lem:JMS_QJ-identities}\ref{QJ20} applied to $e\mu_z\mu_e$, this is}
			&= e\mu_z\mu_e\mu_{x,e}\mu_e\mu_{x,e} + e\mu_z\mu_e\mu_{e\mu_x,e} - e\mu(e\mu_{x,e\mu_z},x) \\
			&= e\mu_z\mu_e\mu_{x,e}\mu_e\mu_{x,e} - e\mu_{e\mu_z,e\mu_x} - e\mu(e\mu_{x,e\mu_z},x)\;,
		\end{align*}
		by Lemma~\ref{lem:JMS_basic_identities}\ref{itm:muxe}. The term $e\mu_{e\mu_z,e\mu_x}$ is again symmetric in $x$ and $z$, so it is sufficient to prove that the remaining difference is symmetric. We will, in fact, show that this expression is always $0$ by using Lemma~\ref{lem:JMS_basic_identities}\ref{itm:muxe} twice:
		\begin{align*}
			e\mu_z\mu_e\mu_{x,e}\mu_e\mu_{x,e} &= -e\mu(e\mu_z\mu_e\mu_{x,e},x) \\
				&= -e\mu(-e\mu_{e\mu_z,x},x) = e\mu(e\mu_{x,e\mu_z},x)\;.
		\end{align*}
		Putting everything together, we get
		\[e\mu(z\mu_e\mu_x,z) = e\mu_{x,z}\mu_e\mu_{x,z} - e\mu_{e\mu_z,e\mu_x}\;,\]
		which is symmetric in $x$ and $z$, hence we must have $e\mu(z\mu_e\mu_x,z) = e\mu(x\mu_e\mu_z,x)$.
		\item Linearize $x$ to $x+v\cdot\ell$ in \ref{QJ29} and take the coefficients of $\ell^1$. This shows the desired identity for any unit $v\in V^+$. If $v$ is not a unit, add the identity for $v-e$ and $e$ (both units) and use the linearity in $v$ to get the identity for $v$.
		\item Start with Lemma~\ref{lem:JMS_QJ-identities}\ref{QJ9} and set $z=e$, $v = v\mu_e$ and $y = z\mu_e$:
		\[z\mu_e\mu(x,v\mu_e\mu_{x,e})+z\mu_e\mu(e,v\mu_e\mu_x) = x\mutil_{z\mu_e,v\mu_e}\mu_{x,e} + e\mutil_{z\mu_e,v\mu_e}\mu_x\;.\]
		Using Lemma~\ref{lem:JMS_basic_identities}\ref{itm:muxe} twice, we get
		\begin{align*}
			z\mu_e\mu(x,v\mu_e\mu_{x,e})
			&= x\mutil_{z\mu_e,v\mu_e}\mu_{x,e} + e\mutil_{z\mu_e,v\mu_e}\mu_x - z\mu_e\mu(e,v\mu_e\mu_x) \\
			&= x\mu_e\mu_{z,v}\mu_e\mu_{x,e} - e\mu_{z,v}\mu_e\mu_x - z\mu_e\mu(v\mu_e\mu_x,e) \\
			&= -e\mu(x\mu_e\mu_{z,v},x) - e\mu_{z,v}\mu_e\mu_x + e\mu(z,v\mu_e\mu_x)\;,
		\intertext{and by \ref{QJ34} with $x$ and $z$ interchanged, we deduce that}
			z\mu_e\mu(x,v\mu_e\mu_{x,e})
            &= - e\mu_{z,v}\mu_e\mu_x - e\mu(z\mu_e\mu_x,v)\;.
		\end{align*}
		Next, we apply $\mu_e\mu_{z,e}$ to this identity, and we use $v\mu_e\mu_{x,e} = x\mu_e\mu_{v,e}$, a consequence of Proposition~\ref{prop:JMS_JP1_units}\ref{itm:JP1_units}:
		\begin{align*}
		e\mu_{z,v}\mu_e\mu_x\mu_e\mu_{z,e} + e\mu(z\mu_e\mu_x,v)\mu_e\mu_{z,e}
		&= -z\mu_e\mu(x,x\mu_e\mu_{v,e})\mu_e\mu_{z,e} \\
		&= e\mu(z,z\mu_e\mu(x,x\mu_e\mu_{v,e}))\;.
		\end{align*}
		Finally, take \ref{QJ34} and set $v = x\mu_e\mu_{v,e}$ (this need not be a unit, but we have shown this identity for non-units as well). This gives
		\[e\mu(z\mu_e\mu(x,x\mu_e\mu_{v,e}),z) = e\mu(x\mu_e\mu_z,x\mu_e\mu_{v,e}) + e\mu(x\mu_e\mu_{v,e}\mu_e\mu_z,x)\;,\]
		hence
		\[e\mu(x\mu_e\mu_z,x\mu_e\mu_{v,e}) + e\mu(x\mu_e\mu_{v,e}\mu_e\mu_z,x) = e\mu_{z,v}\mu_e\mu_x\mu_e\mu_{z,e} + e\mu(z\mu_e\mu_x,v)\mu_e\mu_{z,e}\;.\]
		\item Set $z = e$ and $y = z\mu_e$ in Lemma~\ref{lem:JMS_QJ-identities}\ref{QJ7}. This gives
		\[\mu_e\mu_x\mu_e\mu_z\mu_e+\mu_e\mu_z\mu_e\mu_x\mu_e+\mu_{x,e}\mu_e\mu_z\mu_e\mu_{x,e} = \mu(z\mu_e\mu_{x,e}) + \mu(z\mu_e\mu_x,z)\;.\]
		We take the difference of this identity with the same identity, but interchanging $x$ and $z$. Using $z\mu_e\mu_{x,e} = x\mu_e\mu_{z,e}$, we get
		\[\mu_{x,e}\mu_e\mu_z\mu_e\mu_{x,e}-\mu_{z,e}\mu_e\mu_x\mu_e\mu_{z,e} = \mu(z\mu_e\mu_x,z) - \mu(x\mu_e\mu_z,x)\;,\]
		which we apply to $v\mu_e$:
		\[v\mu_e\mu_{x,e}\mu_e\mu_z\mu_e\mu_{x,e}-v\mu_e\mu_{z,e}\mu_e\mu_x\mu_e\mu_{z,e} = v\mu_e\mu(z\mu_e\mu_x,z) - v\mu_e\mu(x\mu_e\mu_z,x)\;.\]
		We repeatedly use Lemma~\ref{lem:JMS_basic_identities}\ref{itm:muxe} to get
		\begin{align*}
			v\mu_e\mu(z\mu_e\mu_x,z) - v\mu_e\mu(x\mu_e\mu_z,x)
			&= v\mu_e\mu_{x,e}\mu_e\mu_z\mu_e\mu_{x,e} - v\mu_e\mu_{z,e}\mu_e\mu_x\mu_e\mu_{z,e} \\
			&= x\mu_e\mu_{v,e}\mu_e\mu_z\mu_e\mu_{x,e} + e\mu_{v,z}\mu_e\mu_x\mu_e\mu_{z,e} \\
			&= -e\mu(x,x\mu_e\mu_{v,e}\mu_e\mu_z) + e\mu_{v,z}\mu_e\mu_x\mu_e\mu_{z,e} \\
			&= e\mu(x\mu_e\mu_z,x\mu_e\mu_{v,e}) - e\mu(z\mu_e\mu_x,v)\mu_e\mu_{z,e} \\
			&= -e\mu(x\mu_e\mu_z,e\mu_{x,v}) + e\mu(z,e\mu(z\mu_e\mu_x,v))
		\end{align*}
		in which the second last step follows from \ref{QJ33}.
		\item Set $z = v$ and $v = x\mu_e\mu_z$ in Lemma~\ref{lem:JMS_QJ-identities}\ref{QJ27} to get
		\[e\mu(x\mu_e\mu_z,e\mu_{x,v}) - v\mu_e\mu(x\mu_e\mu_z,x) = x\mu_e\mu(x\mu_e\mu_z,v)\;.\]
		Combining this with \ref{QJ32} gives
		\[x\mu_e\mu(x\mu_e\mu_z,v) = e\mu(z,e\mu(z\mu_e\mu_x,v)) - v\mu_e\mu(z\mu_e\mu_x,z)\;.\]
		Next, substituting $z\mu_e\mu_x$ for $z$, $z$ for $v$ and $v$ for $x$ in Lemma~\ref{lem:JMS_QJ-identities}\ref{QJ27} to get
		\[e\mu(z,e\mu(z\mu_e\mu_x,v)) - v\mu_e\mu(z\mu_e\mu_x,z) = z\mu_e\mu_x\mu_e\mu_{z,v}\;.\]
		Combining these last two identities gives us
		\[z\mu_e\mu_x\mu_e\mu_{z,v} = x\mu_e\mu(x\mu_e\mu_z,v)\;.\]
		Substituting $x\mu_e$ for $y$, $x$ for $z$ and $z$ for $v$ turns this into
		\[x\mu_y\mu_{x,z} = y\mu_{y\mu_x,z}\;,\]
		which is the identity we wanted.\qedhere
	\end{enumerate}
\end{proof}

\subsection{The construction gives a Jordan pair}

We can now prove \ref{axiom:JP1} and \ref{axiom:JP2} by linearizing their counterparts for units.

\begin{proposition}\label{prop:JP1}
	In a local Moufang set where \hyperref[itm:J1]{\normalfont{(J1-4)}} holds, we have the following identities:
	\begin{enumerate}
		\item $\{x\,y\,w\mu_{x,z}\} + \{z\,y\,w\mu_x\} = \{y\,x\,w\}\mu_{x,z} + \{y\,z\,w\}\mu_x$
			\begin{flushright}for all units $y,w\in V^-$ and all units $x,z\in V^+$;\end{flushright} \label{itm:JP1almostlinear}
		\item $\{x\,y\,w\mu_{v,z}\} + \{v\,y\,w\mu_{x,z}\} + \{z\,y\,w\mu_{x,v}\} = \{y\,x\,w\}\mu_{v,z} + \{y\,v\,w\}\mu_{x,z} + \{y\,z\,w\}\mu_{x,v}$
			\begin{flushright}for all $y,w\in V^-$ and all $x,z,v\in V^+$;\end{flushright} \label{itm:JP1linear}
		\item $\{x\,y\,w\mu_{x,x}\} = \{y\,x\,w\}\mu_{x,x}$ \quad for all $y,w\in V^-$ and $x\in V^+$.
	\end{enumerate}
\end{proposition}
\begin{proof}\leavevmode
	\begin{enumerate}
		\item After renaming, $\{x\,y\,w\mu_x\} = \{y\,x\,w\}\mu_x$ follows from Proposition~\ref{prop:JMS_JP1_units}\ref{itm:JP1_units}. We linearize $x$ to $x+z\cdot\ell$ in this identity, and the equality of the coefficients of $\ell^1$ is then the desired identity.
		\item We linearize $x$ to $x+v\cdot\ell$ in \ref{itm:JP1almostlinear} and take the coefficients of $\ell^1$ to get the identity we want for units, i.e.
		\[\{x\,y\,w\mu_{v,z}\} + \{v\,y\,w\mu_{x,z}\} + \{z\,y\,w\mu_{x,v}\} = \{y\,x\,w\}\mu_{v,z} + \{y\,v\,w\}\mu_{x,z} + \{y\,z\,w\}\mu_{x,v}\]
		holds for all units $x,z,v\in V^-$ and all units $y,w\in V^+$. We now claim that the variables do not need to be units. If any of the variables $X$ is not a unit, take any unit $e$ and write $X = (X-e)+e$ (or $X = (X\mintil e)\plustil e$ if the variable is in $V^-$). The required identity then follows, using the linearity in $X$, and the fact that the identity holds for the units $X-e$ and $e$. Hence the identity holds for any $x,z,v\in V^-$ and $y,w\in V^+$.
		\item We take $x=z=v$ in \ref{itm:JP1linear}, and hence get
		\[\{x\,y\,w\mu_{x,x}\}\cdot 3 = \{y\,x\,w\}\mu_{x,x}\cdot 3\]
		for any $x\in V^-$ and $y,w\in V^+$. By the unique $3$-divisibility, we get the desired identity.\qedhere
	\end{enumerate}
\end{proof}

\begin{proposition}\label{prop:JP2}
	In a local Moufang set where \hyperref[itm:J1]{\normalfont{(J1-4)}} holds, we have the following identities:
	\begin{enumerate}
		\item $\{v\,x\mu_y\,z\} + \{x\,v\mu_y\,z\} = \{y\mu_{x,v}\,y\,z\}$
			\quad for all units $x,z,v\in V^+$ and any unit $y\in V^-$; \vspace{1ex}
		\item $\{v\,x\mu_{y,w}\,z\} + \{x\,v\mu_{y,w}\,z\} = \{y\mu_{x,v}\,w\,z\} + \{w\mu_{x,v}\,y\,z\}$
			\begin{flushright} for all $x,z,v\in V^+$ and all $y,w\in V^-$; \end{flushright} \label{itm:JP2linear}
		\item $\{x\,x\mu_{y,y}\,z\} = \{y\mu_{x,x}\,y\,z\}$ \quad for all $x,z\in V^+$ and $y\in V^+$.
	\end{enumerate}
\end{proposition}
\begin{proof}\leavevmode
	\begin{enumerate}
		\item Using the definition of the triple product, we can rewrite Proposition~\ref{prop:JMS_JP2_units}\ref{itm:JP2_units} as $\{x\,x\mu_y\,z\} = \{y\mu_x\,y\,z\}$ for all units $x,z\in V^+$ and all units $y\in V^-$. We linearize $x$ to $x+v\cdot\ell$ and take coefficients of $\ell^1$ to get the desired identity.
		\item We linearize $y$ to $y\plustil w\cdottil\ell$ and take the coefficients of $\ell^1$ to get the required identity for units, i.e.
		\[\{v\,x\mu_{y,w}\,z\} + \{x\,v\mu_{y,w}\,z\} = \{y\mu_{x,v}\,w\,z\} + \{w\mu_{x,v}\,y\,z\}\]
		holds for all units $x,z,v\in V^+$ and all units $y,w\in V^-$. As in the proof of \ref{prop:JP1}\ref{itm:JP1linear}, we can use linearity to prove this identity for all $x,z,v\in V^+$ and $y,w\in V^-$.
		\item In \ref{itm:JP2linear}, set $v=x$ and $w=y$ to get
		\[\{x\,x\mu_{y,y}\,z\}\cdot 2 = \{y\mu_{x,x}\,y\,z\}\cdot 2\]
		for any $x,z\in V^+$ and any $y\in V^-$. By the unique $2$-divisibility, we get the desired identity.\qedhere
	\end{enumerate}
\end{proof}

Using these linearizations, we can immediately show that we have a Jordan pair.

\begin{theorem}\label{thm:localJP}
	Let $\M$ be a local Moufang set satisfying \hyperref[itm:J1]{\normalfont{(J1-4)}}.
    Then Construction~\ref{constr:Jpair} gives a Jordan pair $(V^+,V^-)$ with
	\[Q_x^+ = \mu_{x,x}\cdot\frac{1}{2}\text{ for all $x\in V^+$}\qquad\text{and}\qquad Q_y^- = \mutil_{y,y}\cdottil\frac{1}{2}\text{ for all $y\in V^-$.}\]
	Furthermore, the non-invertible elements form a proper ideal $I = (I^+,I^-) = (\overline{0},\overline{\infty})$,
    so $V$ is a local Jordan pair with $\Rad V = I$. Moreover, $V^+$ is uniquely $2$- and $3$-divisible.
\end{theorem}
\begin{proof}
	By \ref{itm:J2}, both $V^+$ and $V^-$ are $\Z$-modules, and by the fact that $\mu$-maps are morphisms between the (abelian) groups $V^+$ and $V^-$, the maps $Q_x^+$ and $Q_y^-$ are homomorphisms. By \ref{itm:J4}, the map $x\mapsto \mu_{x,x}$ is quadratic in $x$ and $y\mapsto\mutil_{y,y}$ is quadratic in $y$. By the Propositions~\ref{prop:JP1} and \ref{prop:JP2} (which also hold when interchanging $+$ and $-$), \ref{axiom:JP1} and \ref{axiom:JP2} hold, along with their linearizations, so by Proposition~\ref{prop:JPsufficientaxioms}\ref{itm:sufficient_JP}, $(V^+,V^-)$ is a Jordan pair.

	Next, we want to prove the Jordan pair is local. We first claim that if $x\in V^\sigma$ is a unit, then it is invertible in the Jordan pair. For such $x$, we have $Q_x^\sigma = \mu_x$, which is an involution and hence invertible with $x^{-1} = x\mu_x = -x$. Next, we show that $I$ is an ideal. If $x\in I^\sigma$, we have $x\mu_z\in I^{-\sigma}$ for any unit $z$. For any $y\in V^{-\sigma}$ the element $xQ_y$ is a linear combination of such $x\mu_z$ by Remark~\ref{rem:expression_muxz}, so $xQ_y\in I^{-\sigma}$. Next, if $x\in I^\sigma$ and $y\in V^{-\sigma}\setminus I^{-\sigma}$, we have $yQ_x = \{x y x\} = \{y\,x\,x\mu_y\}\mu_y\in I^{-\sigma}$, again because this is a linear combination of $x\mu_z\in I^{-\sigma}$. Here we used the fact that $y$ is a unit (so $\mu_y$ is invertible) together with \ref{axiom:JP1}. Finally, if $x\in I^\sigma$, $y\in V^{-\sigma}\setminus I^{-\sigma}$ and $z\in V^\sigma$, we have $\{xyz\} = \{y\,x\,z\mu_y\}\mu_y\in I^\sigma$. Hence $I$ is an ideal. In particular, $I$ does not contain any invertible elements. As all elements of $V\setminus I$ are invertible, $I$ is precisely the set of non-invertible elements; we conclude that $V$ is a local Jordan pair and $I = \Rad V$. The fact that $V^+$ is uniquely $2$- and $3$-divisible is a consequence of \ref{itm:J3}, using \hyperref[itm:J1]{(J1-2)}.
\end{proof}

\section{Back and forth}\label{sec:connect}

In section~\ref{sec:JP} we described a way to create a local Moufang set $\M(V)$ from a local Jordan pair $V$, while section~\ref{sec:LMStoJP} contains a way to construct local Jordan pairs from certain local Moufang sets. Now we will investigate how these two constructions interact.

Suppose first that we start with a local Jordan pair $V$ and apply Theorem~\ref{thm:MV} to obtain a local Moufang set $\M(V)$.
It is natural to ask whether we can apply Theorem~\ref{thm:localJP} to $\M(V)$ in order to retrieve the local Jordan pair $V$.
We begin by verifying that the conditions required to apply this theorem are indeed satisfied.
\begin{proposition}
	Let $V = (V^+,V^-)$ be a local Jordan pair such that $V^+$ is uniquely $2$- and $3$-divisible. Then $\M(V)$ satisfies the conditions \hyperref[itm:J1]{\normalfont{(J1-4)}}
    from Construction~\ref{constr:Jpair}.
\end{proposition}
\begin{proof}
	By the definition of $\M(V)$ we have $[x,0]\alpha_{[v,0]} = [x+v,0]$, so $U_\infty = \{\alpha_x\mid x\in \P(V)\setminus\overline{\infty}\}\cong V^+$. Hence $U_\infty$ is abelian and by Proposition~\ref{prop:gammaJP}, $\M(V)$ is special. Next, if $[x,0]$ is a unit, then $x$ is invertible, which means $Q_x$ is invertible. As $Q_{2x} = Q_x\cdot 4$ and $Q_{3x} = Q_x\cdot 9$, we also have $2x$ and $3x$ invertible, so $[2x,0]$ and $[3x,0]$ are also units. This means \hyperref[itm:J1]{(J1-3)} are satisfied.

	To show \ref{itm:J4}, we compute $\mu_{x,x'}$ for units $x=[v,0]$ and $x'=[v',0]$ such that $x+x'$ is also a unit. By Proposition~\ref{prop:gammaJP}, we have $\mu_{x,x'} = \mu_{v+v'}-\mu_v-\mu_{v'}$. If we apply this to any $[e,e^{-1}+y]$, we get
	\[[e,e^{-1}+y](\mu_{v+v'}-\mu_v-\mu_{v'}) = [yQ_{v+v'},0] - [yQ_v,0]-[yQ_{v'},0] = [yQ_{v,v'},0]\;.\]
	Therefore, we can define $\mu_{x,x'}$ for arbitrary $x$ and $x'$ by
	\[ [e,e^{-1}+y]\mu_{x,x'} := [yQ_{v,v'},0]\;. \]
	As $(x,x')\mapsto Q_{v,v'}$ is bilinear, so is the map $(x,x')\mapsto\mu_{x,x'}$. A similar argument shows that we can define
	\[[x,0]\mutil_{[e,e^{-1}+w],[e,e^{-1}+w']} := [e,e^{-1}+xQ_{w,w'}]\]
	for arbitrary $y=[e,e^{-1}+w]$ and $y'=[e,e^{-1}+w']$, and that $(y,y')\mapsto\mutil_{y,y'}$ is bilinear. Hence \ref{itm:J4} holds.
\end{proof}
We now know that we can apply Theorem~\ref{thm:localJP} on $\M(V)$, so we can compare the resulting local Jordan pair to the original local Jordan pair $V$.

\begin{theorem}
	Let $V = (V^+,V^-)$ be a local Jordan pair with quadratic maps $Q$ such that $V^+$ is uniquely $2$- and $3$-divisible. Denote the local Jordan pair we get from applying Theorem~\ref{thm:localJP} to $\M(V)$ by $W = (W^+,W^-)$. Then $V\cong W$.
\end{theorem}
\begin{proof}
    Denote the quadratic maps of the Jordan pair $W$ by $U$.
	By construction, $W^+ = \{[x,0]\mid x\in V^+\}$ and $W^- = \{[e,e^{-1}+y]\mid y\in V^-\}$. We compute the addition on $W$:
	\begin{align*}
		[x,0]+[x',0] &= [0,0]\alpha_{v}\alpha_{v'} = [v+v',0] \quad \text{and} \\
		[e,e^{-1}+y]\plustil[e,e^{-1}+y'] &= [e,e^{-1}]\gamma_{[e,e^{-1}+y]\tau}\gamma_{[e,e^{-1}+y']\tau} = [e,e^{-1}]\gamma_{[yQ_e,0]}\gamma_{[y'Q_e,0]} \\
			&= [e,e^{-1}]\zeta_{y}\zeta_{y'} = [e,e^{-1}+y+y']\;,
	\end{align*}
	where we used Definition~\ref{def:alphazetaJP} and Proposition~\ref{prop:gammaJP}. A second ingredient we will need, is the actions of the $\mu$-maps. By Proposition~\ref{prop:gammaJP} we have $\mu_{[x,0]} = \mu_{x}$ for all invertible $x\in V^+$. By Proposition~\ref{prop:JordanMuaction} this means $[e,e^{-1}+y]\mu_{[x,0]} = [yQ_x,0]$. Similarly, for all invertible $y\in V^-$, $\mu_{[e,e^{-1}+y]} = \mu_{[-y^{-1},0]} = \mu_{-y^{-1}}$, so
	\[[x,0]\mu_{[e,e^{-1}+y]} = [e,e^{-1}+xQ_{-y^{-1}}^{-1}] = [e,e^{-1}+xQ_y]\;.\]
	We are now ready to define an isomorphism between $V$ and $W$:
	\[h_+\colon W^+\to V^+\colon [x,0]\mapsto x\qquad h_-\colon W^-\to V^-\colon [e,e^{-1}+y]\mapsto y\;.\]
	What remains to be proven is the linearity of these maps, and the fact that they preserve the quadratic maps of the Jordan pairs. Linearity is immediate from our computation of the addition on $W$. Next, take any $[x,0]\in W^+$ and $[e,e^{-1}+y]\in W^-$. If $[x,0]$ is a unit, we have
	\begin{align*}
		h_+([e,e^{-1}+y]U^+_{[x,0]}) &= h_+([e,e^{-1}+y]\mu_{[x,0]}) = yQ_x = h_-([e,e^{-1}+y])Q^+_{h_+([x,0])}\;.
	\end{align*}
	If $[x,0]$ is not a unit, we get
	\begin{align*}
		h_+([e,e^{-1}+y]U^+_{[x,0]}) &= h_+([e,e^{-1}+y]\mu_{[x,0],[x,0]}\cdot\tfrac{1}{2}) \\
			&= h_+([e,e^{-1}+y](\mu_{[e+x,0]}\cdot2-\mu_{[2e+x,0]}+\mu_{[e,0]}\cdot2)) \\
			&= h_+([e,e^{-1}+y](\mu_{e+x}\cdot2-\mu_{2e+x}+\mu_{e}\cdot2)) \\
			&= h_+([yQ^+_{e+x}\cdot2-yQ^+_{2e+x}+yQ^+_{e}\cdot2,0]) \\
			&= y(Q^+_{e,x}+Q^+_e+Q^+_x)\cdot2-y(Q^+_{e,x}\cdot 2+Q^+_x+Q^+_{e}\cdot 4)+yQ^+_{e}\cdot2 \\
			&= yQ^+_x = h_-([e,e^{-1}+y])Q^+_{h_+([x,0])}\;.
	\end{align*}
	Similarly, we get
	\begin{align*}
		h_-([x,0]U^-_{[e,e^{-1}+y]}) &= h_-([x,0]\mu_{[e,e^{-1}+y]}) = xQ^-_y = h_+([x,0])Q^-_{h_-([e,e^{-1}+y])}
	\end{align*}
	for $[e,e^{-1}+y]$ a unit, and otherwise
	\begin{align*}
		h_-([x,0]U^-_{[e,e^{-1}+y]}) &= h_-([x,0]\mutil_{[e,e^{-1}+y],[e,e^{-1}+y]}\cdottil\tfrac{1}{2}) \\
			&= h_-([x,0](\mu_{[e,e^{-1}+e^{-1}+y]}\cdottil2\mintil\mu_{[e,e^{-1}+2e^{-1}+y]}\plustil\mu_{[e,e^{-1}+e^{-1}]}\cdottil2)) \\
			&= h_-([xQ^-_{e^{-1}+y}\cdot 2 - xQ^-_{2e^{-1}+y} + xQ^-_{e^{-1}}\cdot 2]) \\
			&= xQ^-_{e^{-1}+y}\cdot 2 - xQ^-_{2e^{-1}+y} + xQ^-_{e^{-1}}\cdot 2 \\
			&= xQ^-_y = h_+([x,0])Q^-_{h_-([e,e^{-1}+y])}\;.
	\end{align*}
	Hence $(h_+,h_-)$ is a homomorphism from $W$ to $V$, and since it is a bijection, it is also an isomorphism.
\end{proof}
\begin{corollary}
	If $V$ and $W$ are local Jordan pairs and $\M(V)\cong \M(W)$, then $V\cong W$.
\end{corollary}

\medskip

Conversely, suppose now that we start with a local Moufang set $\M$ to which we apply Theorem~\ref{thm:localJP} to get a local Jordan pair $V$, and consider the local Moufang set $\M(V)$ obtained from $V$ by Theorem~\ref{thm:MV}; it is now natural to ask whether $\M \cong \M(V)$.
We will be able to give a positive answer to this question provided that we impose an additional assumption determining the action of $U$ on $\overline{\infty}$.

\begin{theorem}\label{th:extra}
	Let $\M$ be a local Moufang set satisfying \hyperref[itm:J1]{\normalfont{(J1-4)}}, and let $V$ be the local Jordan pair obtained from $\M$ by applying Theorem~\ref{thm:localJP}. Assume that
    \begin{equation}\label{eq:extra}
	    t\alpha_x \mintil x\mutil_{t,t\alpha_x} \plustil t\alpha_x\mu_{x,x}\mutil_{t,t}\cdottil\tfrac{1}{4}
            = t \mintil x\mutil_{t,t}\cdottil\tfrac{1}{2}\qquad\text{for all $t\sim\infty$ and $x\nsim \infty$\;.} \tag{$*$}
    \end{equation}
	Then $\M\cong\M(V)$.
\end{theorem}
\begin{proof}
	To avoid confusion, we will denote the set with equivalence of the local Moufang set $\M$ by $(X,{\sim})$, and the corresponding root group $U_\infty$ by $U$. Recall that $\M(V)$ acts on the set $\P(V)$; we will denote the root group $U_{[0,0]}$ by $U'$. To prove that $\M \cong \M(V)$, we need an equivalence-preserving bijection $\phi\colon X\to\P(V)$, an isomorphism $\theta\colon U\to U'$ and a $\mu$-map in each local Moufang set, which we will denote by $\tau$ and $\tau'$, respectively, such that the action of $U$ and $\tau$ on $X$ are permutationally equivalent with the action of $U'$ and $\tau'$ on $\P(V)$.

	Let $e$ be a unit in $X$; then by \eqref{eq:P(V)} and Theorem~\ref{thm:localJP} we can describe $\P(V)$ as
    \[ \P(V) = \{[t,0]\mid t\nsim\infty\}\cup\{[e,e^{-1}+t]\mid t\sim\infty\}\;. \]
    We define
	\[ \phi \colon X\to \P(V)  \colon  t\mapsto
        \begin{cases}
            [t,0] &\text{if $t\in X\setminus\overline{\infty}$,} \\
            [e,e^{-1}+t] & \text{if $t\in \overline{\infty}$.}
        \end{cases} \]
        We check that this bijection preserves the equivalence, using Definition~\ref{def:radequiv}. First, if $x,x'\nsim\infty$, we have
        \[x\sim x'\iff x-x'\sim 0\iff x-x'\in\Rad V^+\iff [x,0]\sim[x',0] \iff \phi(x)\sim\phi(x')\;.\]
        Second, if $y\sim y'\sim\infty$, then $y,y'\in\Rad V^-$, so $\phi(y)=[e,e^{-1}+y]\sim[e,e^{-1}+y']=\phi(y')$. Finally, if $x\nsim\infty$ and $y\sim\infty$ (or vice-versa) then $[x,0]\nsim[e,e^{-1}+y]$.

        Next, we set $\tau = \mu_e$ and $\tau' = \mu_{[e,0]}$. Then
        \begin{align*}
        	\phi(t\tau) = \phi(t\mu_e) &= \begin{cases}
			[t\mu_e,0] &\text{if $t\nsim 0$}, \\
			[e,e^{-1}+t\mu_e] & \text{if $t\sim 0$};
			\end{cases}\\
        	\phi(t)\tau' &= \begin{cases}
            [t,0]\mu_{[e,0]} = [e,e^{-1}+tQ_e^{-1}] &\text{if $t\sim 0$}, \\
            [t,0]\mu_{[e,0]} = [-t^{-1}Q_e,0] & \text{if $t$ is a unit}, \\
            [e,e^{-1}+t]\mu_{[e,0]} = [tQ_e,0] & \text{if $t\sim\infty$}.
        \end{cases}
        \end{align*}
	Observe that $Q_e = \mu_{e,e}\cdot\frac{1}{2} = \mu_e$, so $Q_e^{-1} = \mu_e^{-1} = \mu_e$, and that $t^{-1} = t\mu_t = -t$ if $t$ is a unit. Hence, in all cases, $\phi(t\tau) = \phi(t)\tau'$.

	We now define
	\[\theta \colon U\to U' \colon \alpha_x\mapsto\alpha_{[x,0]}\qquad\text{for all $x\nsim\infty$.}\]
	This is clearly a group isomorphism. It only remains to verify that $\phi(t\alpha_x) = \phi(x)\theta(\alpha_x)$. If $t\nsim\infty$, we have
	\[\phi(t\alpha_{x}) = \phi(t+x) = [t+x,0] = [t,0]\alpha_{[x,0]} = \phi(t)\theta(\alpha_{x})\;,\]
	so the only case left to consider is when $t\sim\infty$. By~\eqref{eq:extra}, we have
	\begin{alignat*}{2}
		&&t\alpha_x \mintil x\mutil_{t,t\alpha_x} \plustil t\alpha_x\mu_{x,x}\mutil_{t,t}\cdottil\tfrac{1}{4} &= t \mintil x\mutil_{t,t}\cdottil\tfrac{1}{2} \\
		&\implies& t\alpha_x \mintil t\alpha_xD_{t,x} \plustil t\alpha_xQ_xQ_t &= t \mintil xQ_t\\
		&\implies& t\alpha_x(\id \mintil D_{t,x} \plustil Q_xQ_t) &= t \mintil xQ_t\;.
	\end{alignat*}
	Now observe that $(t,x)$ is quasi-invertible because $V$ is a local Jordan pair and $t\in \Rad V^-$; hence $\id \mintil D_{t,x} \plustil Q_xQ_t$ is invertible and
	\begin{align*}
		t\alpha_x &= t \mintil xQ_t(\id \mintil D_{t,x} \plustil Q_xQ_t)^{-1} = t^x\;.
	\end{align*}
	We conclude that also in this case,
	\[\phi(t\alpha_{x}) = \phi(t^{x}) = [t^{x},0] = [t,0]\alpha_{[x,0]} = \phi(t)\theta(\alpha_x)\;.\]
	Hence we have shown that $\M$ and $\M(V)$ are isomorphic.
\end{proof}

\begin{remark}
	As can be observed in the proof, the extra condition~\eqref{eq:extra} is a translation of the original definition of $\alpha_x$ in $\M(V)$: $[e,e^{-1}+t]\alpha_x = [e,e^{-1}+t^x]$. It is at this point unclear to us whether this assumption is strictly necessary.
    It seems likely that there is a connection with the extra assumption needed in \cite[Theorem~5.20]{DMRijckenPSL}, but we have not been able to verify this.
\end{remark}

\bibliographystyle{alpha}
\bibliography{universal}

\end{document}